\documentclass[11pt, leqno]{article}
\usepackage[utf8]{inputenc}
\usepackage{lmodern}
\usepackage{subfiles}
\usepackage{enumitem}
	\setenumerate{label={\normalfont(\alph*)}, itemsep=0em} 

\usepackage{amsfonts}
\usepackage{amsthm}
\usepackage{amsmath}
\usepackage{amssymb}
\usepackage{amscd}
\usepackage{mathrsfs}
\usepackage{mathtools}
\usepackage{esint}

\usepackage[margin=3cm]{geometry}
\usepackage{titlefoot}
\usepackage{indentfirst}
\usepackage{graphicx}
\usepackage{graphics}
\usepackage{lscape}
\usepackage{tikz-cd}
\usepackage{color}
\usepackage{pict2e}
\usepackage{epic}
\usepackage{epstopdf}
\usepackage{titlesec}
	\titleformat{\section}[block]{\Large\bfseries\filcenter}{\thesection}{1em}{}
\usepackage{commath}
\usepackage{float}
\usepackage{caption}
\usepackage{etoolbox}
\usepackage[affil-it]{authblk}
\usepackage{combelow}

\let\oldbibliography\thebibliography
\renewcommand{\thebibliography}[1]{%
  \oldbibliography{#1}%
  \setlength{\itemsep}{-.5pt}%
}

\usepackage[hidelinks]{hyperref}
\hypersetup{bookmarksopen=true} 
\usepackage{hypcap}

\graphicspath{{./Pictures/}}
\allowdisplaybreaks

\usepackage{listings}
\usepackage{fancybox}
\makeatletter
\newenvironment{CenteredBox}{%
\begin{Sbox}}{
\end{Sbox}\centerline{\parbox{\wd\@Sbox}{\TheSbox}}}
\makeatother

\theoremstyle{plain}
\newtheorem{bigthm}{Theorem}

\renewcommand*\thesection{\arabic{section}}
\swapnumbers
\numberwithin{equation}{section}

\theoremstyle{plain}
\newtheorem{teo}[equation]{Theorem}
\newtheorem{lema}[equation]{Lemma}
\newtheorem{prop}[equation]{Proposition}
\newtheorem{cor}[equation]{Corollary}

\theoremstyle{definition}
\newtheorem{ndef}[equation]{Definition}
\newtheorem{ex}[equation]{Example}
\newtheorem{question}[equation]{Question}

\newtheorem{remark}[equation]{Remark}

\newcommand{\thistheoremname}{}
\newtheorem{genericthm}[equation]{\thistheoremname}

\newcommand{\thistheoremnames}{}
\newtheorem*{genericthms}{\thistheoremnames}
\newenvironment{para*}[1]
  {\renewcommand{\thistheoremnames}{#1}%
   \begin{genericthms}}
  {\end{genericthms}}

\expandafter\let\expandafter\oldproof\csname\string\proof\endcsname
\let\oldendproof\endproof
\renewenvironment{proof}[1][\proofname]{%
  \oldproof[\upshape \bfseries #1:]%
}{\oldendproof}

\makeatletter
\def\@makechapterhead#1{%
  \vspace*{50\p@}%
  {\parindent \z@ \raggedright \normalfont
    \interlinepenalty\@M
    \Huge\bfseries  \thechapter.\quad #1\par\nobreak
    \vskip 40\p@
  }}
\makeatother

\def \vazio{\emptyset}
\def \a{\alpha}
\def \R {\mathbb{R}}

\def \N{\mathbb{N}}
\def \D{\textup{D}}
\def \e{\varepsilon}
\def \d{\,\textup{d}}
\def \t{\Delta}

\def \exc{\backslash}
\def \p{\partial}

\def \mc{\mathcal}

\def \wstar{\overset{\ast}{\rightharpoonup}}
\def \w{\rightharpoonup}
\def \supp{\textup{supp}\,}
\def \BMO{\textup{BMO}}
\def \hs{\hspace{0.5cm}}

\def \tp{\textup}
\def \mb{\mathbb}

\makeatletter
\DeclareFontFamily{OMX}{MnSymbolE}{}
\DeclareSymbolFont{MnLargeSymbols}{OMX}{MnSymbolE}{m}{n}
\SetSymbolFont{MnLargeSymbols}{bold}{OMX}{MnSymbolE}{b}{n}
\DeclareFontShape{OMX}{MnSymbolE}{m}{n}{
    <-6>  MnSymbolE5
   <6-7>  MnSymbolE6
   <7-8>  MnSymbolE7
   <8-9>  MnSymbolE8
   <9-10> MnSymbolE9
  <10-12> MnSymbolE10
  <12->   MnSymbolE12
}{}
\DeclareFontShape{OMX}{MnSymbolE}{b}{n}{
    <-6>  MnSymbolE-Bold5
   <6-7>  MnSymbolE-Bold6
   <7-8>  MnSymbolE-Bold7
   <8-9>  MnSymbolE-Bold8
   <9-10> MnSymbolE-Bold9
  <10-12> MnSymbolE-Bold10
  <12->   MnSymbolE-Bold12
}{}

\let\llangle\@undefined
\let\rrangle\@undefined
\DeclareMathDelimiter{\llangle}{\mathopen}%
                     {MnLargeSymbols}{'164}{MnLargeSymbols}{'164}
\DeclareMathDelimiter{\rrangle}{\mathclose}%
                     {MnLargeSymbols}{'171}{MnLargeSymbols}{'171}
\makeatother

\newcommand{\mres}{\mathbin{\vrule height 1.6ex depth 0pt width
0.13ex\vrule height 0.13ex depth 0pt width 1.3ex}}

\begin{document}

 \title{\LARGE \textbf{Quasiconvexity, null Lagrangians, and Hardy space integrability under constant rank constraints}}

\author[1]{{\Large Andr\'e Guerra}}
\author[2]{{\Large Bogdan Rai\cb{t}\u{a}}}

\affil[1]{\small University of Oxford, Andrew
  Wiles Building, Woodstock Rd, Oxford OX2 6GG,
  United Kingdom
  \protect\\
  {\tt{guerra@maths.ox.ac.uk}}
  \vskip.2pc \
}

\affil[2]{\small 
Max Planck Institute for Mathematics in the Sciences, Inselstraße 22, 04103 Leipzig, Germany  \protect\\
  {\tt{raita@mis.mpg.de}} 
}

\date{}

\maketitle

\vspace{-0.65cm}
\begin{abstract}
We present a systematic treatment of the theory of Compensated
Compactness under Murat's constant rank assumption. We give a short
proof of a sharp weak lower semicontinuity result for signed
integrands, extending aspects of the results of Fonseca--M\"uller. The null
Lagrangians are an important class of signed integrands, since they
are the weakly continuous functions. We show that they are precisely
the compensated compactness quantities with Hardy space integrability,
thus proposing an answer to a question raised by
Coifman--Lions--Meyer--Semmes. Finally we provide an effective way of
computing the null Lagrangians associated with a given operator.
\end{abstract}

\unmarkedfntext{
\hspace{-0.85cm} 

}

\vspace{0.2cm}

\section{Introduction}

Let $\mc A$ be a linear partial differential operator acting on fields $v\colon \R^n\to \mb V$, for some finite-dimensional inner product space $\mb V$.
In this paper, we address the following question:
\begin{para*}{Main question}
Are there special quantities $F\colon \mb V \to \R$ which are well-behaved with respect to solutions of the system $\mc A v=0$? In particular:
\begin{itemize}[itemsep=0pt, leftmargin=0.5cm]
\item For solutions of $\mc A v=0$, does $F(v)$ benefit from \textbf{compensated regularity}?,  e.g.\
\begin{equation}
	v\in C^\infty_{c}(\R^n, \mb V) \tp{ and } \,\mc A v=0
	\hs \implies \hs F(v)\in \mathscr
	H^1(\R^n)\label{eq:hardyimpintro}.
\end{equation}
\item For solutions of $\mc A v=0$, does $F(v)$ benefit from \textbf{compensated compactness}?, e.g.\
\begin{equation}
v_j\wstar v \tp{ in } \mathscr D'(\R^n,\mb V) \tp{ and } \mc A v_j=0 \hs \implies \hs F(v_j)\wstar F(v)\tp{ in } \mathscr D'(\R^n).\label{eq:CC_intro}
\end{equation}	
\end{itemize}
If there are such quantities, how do we characterize and compute them?
\end{para*} 
It is clear that, for the first part, one should look for \emph{nonlinear} quantities, since otherwise $F(v)$ has precisely the same regularity as $v$. In \eqref{eq:hardyimpintro}, $\mathscr H^1(\R^n)$ denotes the real Hardy space, which can be thought of as a proper subspace of $L^1(\R^n)$ whose elements have cancellations at all scales and, therefore, have an additional logarithm of integrability. 
Being able to identify $L^1$-quantities that in fact have Hardy space
integrability is often important in PDE: this has been useful in Fluid
Dynamics \cite{Evans1994a,Faraco2019,Faraco2019a} as well as Differential Geometry \cite{Helein2002,Muller1995} and we refer the reader to \cite{Coifman1993} for further examples and references.

Weakly continuous functions, as in \eqref{eq:CC_intro}, can be thought of as representing physical quantities that are robust to errors in measurements induced from small-scale oscillations. We call these quantities  \textbf{null Lagrangians} \cite{Ball1977} or $\mc A$-quasiaffine functions \cite{Dacorogna2007} and they are the classical objects of study in the \textsc{Murat}--\textsc{Tartar} theory of Compensated Compactness \cite{Murat1978,Murat1981,Tartar1979,Tartar1983}. In the last four decades,  the theory was developed much further, having found applications in Continuum Mechanics
\cite{DiPerna1985,DiPerna1987,Evans1994a}, Homogenization \cite{Braides2000,Briane2016,Milton1990,Milton2002} and Nonlinear Analysis
\cite{Baia2013,DePhilippis2016,Fonseca2004,Joly1995, Muller2003}. We also refer the reader to the recent papers \cite{Arroyo-Rabasa2018,Conti2019,Davoli2018,Prosinski2018}.

\bigskip

To be precise, in our main question we consider an operator $\mc A$ of the form
$$\mc A=\sum_{|\a|=l} A_\a \p^\a, \qquad \tp{where }A_\a \in \tp{Lin}(\mb V, \mb W),$$
for some finite dimensional inner product spaces $\mb V,\mb W$.
The prototypical example we have in mind is $\mc A=(\tp{div}, \tp{curl})$. For a domain $\Omega\subset \R^n$ and fields $E,\,B\colon \Omega\to \R^{n}$ in $L^2(\Omega)$, which we think of as the electric and the magnetic fields respectively, \textsc{Coifman}--\textsc{Lions}--\textsc{Meyer}--\textsc{Semmes} \cite{Coifman1993} proved that \eqref{eq:hardyimpintro} holds, i.e.\
\begin{equation}
\label{eq:divcurlhardy}
\tp{div}\,B=0, \tp{ curl\,}E=0 \quad \implies \quad B\cdot E\in \mathscr H^1(\R^n).\end{equation}
The implication \eqref{eq:divcurlhardy} was inspired by a surprising and remarkable result of
\textsc{M\"uller} \cite{Muller1990} and it can be proved through the
\textsc{Coifman}--\textsc{Rochberg}--\textsc{Weiss} commutator theorem
\cite{Coifman1976}, see also \cite{Lenzmann2018} for a different
approach and \cite{Dafni2005} for local, non-homogeneous versions.
The quantity $E\cdot B$ is also weakly continuous for the system $(\tp{div}, \tp{curl})$, a fact which goes back to the pioneering work of \textsc{Murat} and \textsc{Tartar} \cite{Tartar1979}:  if $d=n^2+1$, then \eqref{eq:CC_intro} holds, i.e.\
\begin{equation}
\label{eq:divcurlwc}
\begin{rcases}
(B_j,E_j)\w (B,E) & \tp{in } L^2(\Omega,\R^{2n})\\
(\tp{div\,} B_j,\tp{curl\,}E_j)\to (\tp{div\,} B,\tp{curl\,}E) & \tp{in } H^{-1}_\tp{loc}(\Omega,\R^d)
\end{rcases}
\implies B_j\cdot E_j \wstar  B\cdot E \tp{ in } \mathscr D'(\Omega).
\end{equation}

Continuum Mechanics furnishes plenty of interesting examples beyond electromagnetism: 
in the theory of elasticity  the deformation gradient is irrotational, while the linearized strain satisfies the Saint-Venant compatibility condition, and in incompressible fluid flow the velocity field is divergence-free; see also Example \ref{ex:intro}. In these examples the operator $\mc A$ has an important non-degeneracy property:
\begin{equation}
\label{eq:assumption}
\mc{A} \tp{ has constant rank} \hs \tp{and} \hs \tp{span}\, \Lambda_{\mc A}=\mb V.
\end{equation}
Here $\Lambda_{\mc A}$ is the wave cone of the operator $\mc A$, see also Section \ref{sec:CR} for notation and terminology, and
the spanning assumption is natural since
weakly continuous quantities are completely unconstrained along directions
not in $\tp{span}\, \Lambda_\mc A$. The constant rank assumption is standard \cite{Fonseca1999,Murat1981} and, per the results of the authors in \cite{GuerraRaita2020}, it is equivalent to a certain $L^p$-estimate on which many results in compensated compactness theory crucially rely. In the constant rank case, weak continuity is well understood since \textsc{Murat}'s work 
\cite{Murat1981} but, without this assumption, very little is known, an important exception being the case of separate convexity \cite{Muller1999b,Tartar1993}, which was proposed by \textsc{Tartar} as a toy model for rank-one convexity. The case of quadratic functions $F$ is also special: in this setting, there is a satisfactory theory both for Hardy integrability \cite{Coifman1993,Li1997} and for weak continuity \cite{Tartar1979}. 

\medskip

Returning to the div-curl example, we observe that the inner product is both weakly continuous and has Hardy space integrability. Hence, the following natural question was asked in \cite{Coifman1993}: is this a general phenomenon, i.e.\ is it the case that \eqref{eq:hardyimpintro} and \eqref{eq:CC_intro} are \emph{equivalent}? Our main theorem shows that, under the standard assumption \eqref{eq:assumption}, this is indeed the case:


\begin{bigthm}[Hardy integrability equals weak continuity]
\label{teo:hardyintro}
Assume \eqref{eq:assumption} and let $F\colon \mb V\to \R$ be a  locally bounded, Borel function that is not affine. Then \eqref{eq:hardyimpintro} holds if and only if \eqref{eq:CC_intro} holds and in that case we have:
\begin{itemize}[itemsep=0pt, topsep=5pt]
\item $F$ is a polynomial of degree $s\leq \min\{n,\dim \mb V\}$ and it is $\mc A$-quasiaffine, i.e.\ $F$ and $-F$ are both $\mc A$-quasiconvex;
\item if moreover $F$ is homogeneous, there is an estimate
$$\Vert F(v)\Vert_{\mathscr H^1(\R^n)} \leq C\Vert v \Vert_{L^s(\R^n)} \qquad \tp{for all } v\in L^s(\R^n) \tp{ such that } \mc A v=0 \tp{ in } \mathscr D'(\R^n).$$
\end{itemize}
The class of such polynomials can be computed explicitly by solving an algebraic system of linear equations.
\end{bigthm}

Theorem \ref{teo:hardyintro} shows that compensated
compactness and compensated regularity are two facets of the algebraic cancellations in the nonlinearity, which compensate the lack of ellipticity of the operator $\mc A$.

When $F$ is linear, it is possible to make a statement similar to the one in Theorem \ref{teo:hardyintro},  c.f.\ Theorem \ref{teo:hardy}, although we show that there is no estimate in that case. We would also like to highlight that we provide an
effective way of computing the $\mc A$-quasiaffine functions.
\textsc{Murat} \cite{Murat1981} derived the algebraic identity (\ref{eq:derivatives}) that
characterizes these functions but, as he was already
aware, it is in general very difficult to decide
which nonlinear polynomials, if any, satisfy this identity. In order
to  deal with this issue, we  crucially rely on the work of
\textsc{Ball}--\textsc{Currie}--\textsc{Olver} \cite{Ball1981}.
We deduce that all $\mc A$-quasiaffine functions can be written as coefficients of differential forms, which answers in the positive a question of \textsc{Robbin}--\textsc{Rogers}--\textsc{Temple} \cite[\S 5]{Robbin1987} under the assumption (\ref{eq:assumption}).

To prove  Theorem
\ref{teo:hardyintro} we rely on ideas appearing in the literature in specific
instances, typically for the operators $\mc A=\tp{curl}$ or $\mc
A=(\tp{div}, \tp{curl})$  \cite{Coifman1993,Grafakos1992,Lindberg2017}, as well as new techniques that we introduce. Our main new tool is an $L^p$ Helmholtz--Hodge decomposition for constant rank operators, which is based on the existence of potential operators. These were constructed recently by the
second author in \cite{Raita2018}. 

\medskip

In the setup of Theorem \ref{teo:hardyintro}, it is natural to wonder whether the convergence in \eqref{eq:CC_intro} can be improved. \textsc{Tartar}
\cite[Lemma 7.3]{Tartar2010} showed that one cannot upgrade
weak-$*$ convergence in measures to weak convergence in
$L^1$, i.e.\ one cannot test the convergence against $L^\infty$
functions. However, as a by product of Theorem \ref{teo:hardyintro}, one can test the convergence against functions in
$\tp{VMO}(\R^n)$; this a space which is neither contained nor contains $L^\infty(\R^n)$:

\begin{bigthm}[Improved and quantified convergence]\label{teo:nulllagintro}
As before assume \eqref{eq:assumption} and let $F\colon \mb V\to \R$ be $\mc A$-quasiaffine and $s$-homogeneous for some $s\geq 2$. Then
	\begin{equation}
	v_j \w v  \tp{ in } L^s(\R^n,\mb V)\tp{ and } \mc A v_j=0\hs \implies \hs F(v_j)\wstar F(v) \tp{ in } \mathscr H^1(\R^n).\label{eq:hardyimpintro2}
	\end{equation} 
Moreover, let $p\in (s-1,\infty)$ and $q\in (1,\infty)$ be such that  
	$\frac{s-1}{p} + \frac 1 q=1$. For $\mc A$-free fields $v_1, v_2 \in C^\infty_{c}(\R^n, \mb V)$ and any $\varphi \in C^\infty_c(\R^n)$ we have the uniform estimate
	$$
	\left|\int_{\R^n} \varphi\left(F(v_1)-F(v_2)\right)\d x\right|\leq C \Vert v_1 - v_2\Vert_{\dot W^{-1,q}} \left(\Vert v_1\Vert_{L^p} + \Vert v_2 \Vert_{L^p} \right)^{s-1} \Vert \D \varphi \Vert_{L^\infty}.
	$$
\end{bigthm}

The last part of Theorem \ref{teo:nulllagintro} generalizes the quantitative statements in the $\mc A=\tp{curl}$ case of \cite{Brezis2011b} and \cite[\S 8]{Iwaniec2001}, see also 
\cite{Fonseca2005,Iwaniec2002b,Iwaniec1992a}. It shows that, under weaker integrability hypothesis, \emph{distributional} $\mc A$-quasiaffine quantities are still weakly continuous, c.f.\ Section~\ref{sec:estimates} and \cite{GuerraRaitaSchrecker2020}.

\medskip

We conclude this introduction by discussing the more general class of $\mc A$-quasiconvex functions and their weak lower semicontinuity properties. Due to Theorems \ref{teo:hardyintro} and \ref{teo:nulllagintro}, where the functions are polynomials, we are interested in the general case of signed integrands.
This case is not covered by the influential work of \textsc{Fonseca}--\textsc{M\"uller} \cite{Fonseca1999} (see also \cite{Fonseca2010}), where only positive integrands are studied. When the integrand changes sign one needs to deal with the possibility of concentrations of the
sequence on the boundary of the domain. When this happens, weak lower
semicontinuity breaks down: this is already the case when
$\mc A=\tp{curl}$, as an example due to \textsc{Tartar} shows
\cite{Ball1984}. As a consequence, the convergence should be tested against functions which vanish on the boundary.
In Section \ref{sec:lsc}, we prove the following result:

\begin{bigthm}[Weak lower semicontinuity]\label{teo:lscintro}
Let $\Omega\subset \R^n$ be a bounded domain, $p\in (1,\infty)$, and
let $F\colon \mb V\to \R$ be an $\mc A$-quasiconvex function such that
$|F(v)|\leq C(|v|^p+1)$. Then, for all $\varphi \in C^\infty_c(\Omega)$ with $\varphi\geq 0$,
$$\begin{rcases}
v_j \w v &\tp{in }L^p(\Omega,\mb V)\\ 
\mc A v_j \to \mc A v &\tp{in } W^{-l, p}_\tp{loc}(\Omega, \mb V)
\end{rcases}
\quad
\implies
\quad
\liminf_{j\to \infty} \int_\Omega \varphi F(v_j) \d x \geq \int_\Omega \varphi F(v) 
\d x.
$$
This is sharp in the sense that $\varphi$ cannot be taken to be in the space
$C^\infty(\overline \Omega)$.
\end{bigthm}

The methods used to prove Theorem
\ref{teo:lscintro} are distinct and more elementary that the ones from
\cite{Fonseca1999} and, in particular, we avoid the use of 
Young measure machinery. The proof also extends easily to the more general situation of Carath\'eodory
integrands as in \cite{Acerbi1984}.

Due to its relation to weak lower semicontinuity and to the Direct Method, as evidenced by the
above theorem, quasiconvexity is the natural mathematical assumption
on the integrands in the classical curl-free case of the Calculus of Variations
\cite{Ball1977,Dacorogna1982,Muller1999a}. In this context, quasiaffine functions
play an important role in the study of quasiconvexity, for instance
through the notion of polyconvexity; it turns out, however, that in our more general setting there are several distinct competing notions of polyconvexity, see Section \ref{sec:nullLags}. The concept of quasiconvexity is
still poorly understood and the most important question
concerning it is whether it admits an explicit description and, in particular, whether it agrees
with rank-one convexity in $\R^{N\times 2}$ for $N\geq 2$. This last question is
known as Morrey's problem and it remains an
outstandingly difficult problem
\cite{Faraco2008,Grabovsky2018,Guerra2018,Kirchheim2016,Kirchheim2008,Muller1999b,Sverak1992a} with far-reaching
consequences in analysis \cite{Iwaniec2002}. Advances in this
direction have been made through the study of quasiaffine integrands in the more general $\mc
A$-free setup: Morrey's problem was solved---in sufficiently high
dimensions---much earlier for higher order gradients \cite{Ball1981}
than for first order gradients \cite{Sverak1992a}. Furthermore, \textsc{\v Sver\'ak}'s example has many similarities with an older example of
\textsc{Tartar} \cite{Tartar1979} of a $\Lambda_\mc A$-affine
integrand which is not $\mc A$-quasiaffine, where $\mc
A
u=\left(\partial_1u_1,\,\partial_2u_2,\,(\partial_1+\partial_2)u_3\right)$ for $u\colon\R^2\rightarrow\R^3$. It is therefore
interesting to study weak continuity and lower semicontinuity in a
larger class of operators \cite{Dacorogna1982} and the constant rank
assumption is adequate in so far as all constant rank operators are
``curl-like'', in the sense that one can find a potential operator which plays the role of the gradient. 

\subsection*{Outline}
Finally let us give a brief outline of the paper. In Section \ref{sec:prelims} we gather notation as well as basic results that we will use throughout the paper. In Section \ref{sec:CR} we present a systematic treatment of constant rank operators as well as some basic facts concerning cocanceling operators. 
Section \ref{sec:lsc} is dedicated to quasiconvexity and to the lower-semicontinuity proofs while, in Section \ref{sec:nullLags}, we use these results to give both abstract and concrete characterizations of null Lagrangians. In Section \ref{sec:hardy} we study the Hardy space integrability of null Lagrangians and finally in Section \ref{sec:estimates} we prove the quantitative estimates of Theorem \ref{teo:nulllagintro}.

{\paragraph{Acknowledgements.} The authors thank Jan Kristensen for
  continuous support and many insightful comments. We also thank François Murat for a most interesting discussion around the topic of the paper, as well as the origins of its topic. We thank the anonymous
  referees who made numerous suggestions that improved the quality of the
  script.
  A.G. also thanks Federico Franceschini and Miguel M. Santos for helpful discussions.
A.G. was supported by the Engineering and Physical Sciences Research Council [EP/L015811/1]. B.R. received funding from the European Research Council (ERC) under the
European Union’s Horizon 2020 research and innovation programme under grant agreement
No 757254 (SINGULARITY).
}


\section{Preliminaries}\label{sec:prelims}

We begin by fixing some notation that will be used throughout the paper. 
As usual, $\Omega \subseteq \R^n$ will denote an open, bounded set and, unless stated otherwise, $1<p<\infty$.
The letters $\mb U, \mb V, \mb W$ will denote finite-dimensional inner product spaces and, if $\mb U\subset \mb V$, then $\tp{Proj}_\mb U\colon \mb V \to \mb U$ denotes the orthogonal projection onto $\mb U$. The sphere in $\mb V$ is denoted by $S_\mb V$.
We write $\odot^k(\R^n, \mb U)$ for the space of all $\mb U$-valued symmetric $k$-linear maps on $\R^n$; for a $C^k$ map $u\colon \Omega\to \mb U$ we have that $\D^k u \in \odot^k(\R^n, \mb U)$. The notation $\mc M(\Omega)$ denotes the space of Radon measures in $\Omega$.

\subsection{Moore--Penrose generalized inverses}

Let $A\in \tp{Lin}(\mb V, \mb W)$.
We will use the notation $A^\dagger\equiv (A^*A)^{-1}A^*$ if $\ker A=\{0\}$, where $A^*$ denotes the adjoint (transpose) of $A$. In particular, for injective linear transformations between finite-dimensional inner product spaces, we obtain a formula for a left-inverse.
In more generality, the Moore-Penrose generalized inverse of $A$ (which we will here call simply the \textbf{pseudo-inverse}, though this terminology is not standard; algebraists use various algorithms to invert non-square matrices) is defined geometrically as the unique 
$A\in \tp{Lin}(\mb V, \mb W)$ such that 
\begin{equation}A A^\dagger = \tp{Proj}_{\tp{im\,} A} \tp{ and } A^\dagger A= \tp{Proj}_{ \tp{im\,} A^* },\label{eq:pseudoinv}\end{equation} where the projections are orthogonal, see \cite{Campbell2009}. Equivalently, a computable formula is given using the fact that the linear map $A|_{(\ker A)^\bot}\colon (\ker A)^\bot\to \tp{im\,}A$ is bijective. In this case, it is easy to check that
$$A^\dagger \equiv \begin{cases}
(A|_{\ker A^\bot})^{-1} & \tp{on im\,}A\\
0 & \tp{on (im\,}A)^\bot
\end{cases}$$
defines a matrix that indeed satisfies \eqref{eq:pseudoinv}.
We have the following useful fact, c.f.\ \cite{GuerraRaita2020}:

\begin{lema}\label{lema:CRbdd}
Let $\Omega\subset \R^n$ be open. A smooth map $A\colon\Omega\to \tp{Lin}(\mb V, \mb W)$,  $A^\dagger\colon \Omega\to \tp{Lin}(\mb W, \mb V)$ is locally bounded if and only if $\tp{rank}\,A$ is constant in $\Omega$. In that case, $\mc A^\dagger$ is also smooth.
\end{lema}

\subsection{Harmonic Analysis}

In this paper we  only use standard results from Harmonic Analysis, such as the Maximal Theorem and the H\"ormander--Mihlin multiplier theorem, which can be found for instance in the  book  \cite{Stein2016}.
Here we briefly recall some definitions for the convenience of the reader.

Fix a function $\phi \in C_c^\infty(\R^n)$ with non-zero mean and as usual let $\phi_t(x)\equiv t^{-n} \phi(x/t)$ for $t>0$. The \textbf{Hardy space} is defined as
$$\mathscr H^1(\R^n)\equiv \{f \in \mathscr S'(\R^n): \sup_{t>0} |f*\phi_t|\in L^1(\R^n)\}$$
and this definition is independent of the choice of $\phi$ \cite{Fefferman1972}. 
Other characterizations of the Hardy space are possible, for instance through the atomic decomposition. Another possibility, which is relevant for our purposes, is the following (see \cite[III.4.3]{Stein2016}):
\begin{prop}\label{prop:Riesztransfcharact}
	Let $f$ be a distribution which is restricted at infinity in the sense that, for all $r<\infty$ sufficiently large,
	$$f * \varphi \in L^r(\R^n) \hs \tp{ for all } \varphi \in \mathscr S(\R^n).$$
	Then $f\in \mathscr H^1(\R^n)$ if and only if both $f$ and $R_jf$, for
	$j=1, \dots, n$, are in $L^1(\R^n)$, where  $R_j$ is the $j$-th Riesz transform, i.e., $\widehat{R_jf}(\xi)=\xi_j/|\xi|\hat{f}(\xi)$ for $f\in\mathscr{S}(\R^n)$, $\xi\in\R^n$.
\end{prop}

We will also use repeatedly the well-known fact that functions in the Hardy space have zero mean; in fact, a bounded, compactly supported function $f$ is in $\mathscr H^1(\R^n)$ if and only if $\int_{\R^n} f(x) \d x =0$. 

Weak convergence in the Hardy space is induced from its dual, the space $\BMO(\R^n)$ of functions of \textbf{bounded mean oscillation} \cite{Fefferman1971}, defined as the space of those locally integrable functions $f\in L^1_\tp{loc}(\R^n)$ such that
$$\Vert f \Vert_\BMO \equiv \lim_{\delta \to \infty} M_\delta(f)<\infty,\hs \tp{ where }M_\delta(f)\equiv
\sup_{|B|<\delta}\fint_B \left|f-\fint_B f\right |\d x$$
and the supremum runs over balls in $\R^n$.
Here, and in the sequel, we write $\fint_E f \d x\equiv \frac{1}{|E|}\int_E f\d x.$
Moreover, $\mathscr H^1(\R^n)$ is a dual space itself: it is the dual of the space $\tp{VMO}(\R^n)$ of functions of \textbf{vanishing mean oscillation} \cite{Coifman1977,Sarason1975}; this is the space of those functions in $\BMO(\R^n)$ such that
$$\lim_{\delta \to 0} M_\delta(f)=0.$$
In particular, there is a notion of weak-$*$ convergence in $\mathscr H^1$, defined by testing against functions in $\tp{VMO}(\R^n)$. We have the following classical result \cite{Jones1994}:
\begin{teo}[Jones--Journ\'e]\label{teo:jonesjourne}
	If a sequence $f_j$ is bounded in $\mathscr H^1(\R^n)$ and it converges a.e. to $f$ then $f\in \mathscr H^1$ and in fact $f_j\wstar f$ in $\mathscr H^1$.
\end{teo}

Notice that if we replace $\mathscr H^1(\R^n)$ bounds by $L^1(\R^n)$
bounds then the conclusion of the theorem does not hold; in this case,
we have that
\begin{gather*}
\tp{assuming that } f_j \to f \tp{ a.e.,}\\
  f_j \w f \tp{ in } L^1(\R^n) \hs \iff \hs (f_j) \tp{ is
    equi-integrable}.
\end{gather*}
The difference between $\mathscr H^1$ and $L^1$ convergence will be used crucially in  Lemma \ref{lema:Afreefieldshardy} below.

\section{Constant rank linear operators}\label{sec:CR}

Let us consider a collection of linear operators $A_\a\in \tp{Lin}(\mb
V, \mathbb W)$ for each $n$-multi-index $\a$. We define a homogeneous $l$-th order linear operator $\mc A$ by 
\begin{equation}
\label{eq:defa}\mc A v = \sum_{|\a|=l} A_\a \p^\a v, \hs v\colon \Omega\subseteq
\R^n \to \mb V.\end{equation}
We think of $\mc A$ as a polynomial in $\p$ and so we write
$$\mc{A}\colon \R^n \to \tp{Lin}(\mb V, \mathbb W), \hs \mc{A}(\xi)=\sum_{|\a|=l} A_\a \xi^\a.$$
Associated with $\mc A$ we have a set of directions and frequencies, introduced by \textsc{Murat} and \textsc{Tartar} \cite{Murat1978,Tartar1979},
$$V_{\mc A} \equiv \left\{(\lambda, \xi)\in \mb V\times \R^n\exc\{0\}:  \mc A(\xi)\lambda=0\right\}$$
and its projection onto $\mb V$ is the wave cone associated to $\mc A$ which we denote by
$$\Lambda_{\mc A} \equiv \bigcup_{\xi \in 
	\mb S^{n-1}} \ker \mc{A}(\xi).$$ 
We will sometimes drop the subscript $\mc A$ in the above notation.

We say that the operator
$\mc A$ has \textbf{constant rank}
if there is a number $r\in \N$ such that
$$\tp{rank}\, \mc A(\xi)=r \hs \tp{for all } \xi \in \mb S^{n-1}.$$
A geometric interpretation of this property is that $V_\mc A$ is a smooth vector bundle over $\mb S^{n-1}$ with fiber $ \ker \mc A(\xi)$ at the point $\xi$. A more analytic interpretation, c.f.\ Lemma \ref{lema:CRbdd} and \cite{Kato1975,Murat1981,Schulenberger1971}, is the following:

\begin{lema}\label{lema:CRgivesmultiplier}
	The operator $\mc A$ has constant rank if and only if the map
$\xi \mapsto \tp{Proj}_{\ker \mc A(\xi)}$, defined for $\xi\in \R^n\exc\{0\}$,  is bounded. In this case, the map is smooth away from zero.
\end{lema}

Lemma \ref{lema:CRgivesmultiplier} can be used to prove a more refined characterization of constant rank operators. For $\varphi \in C^\infty_c(\R^n,\mb V)$, we write
$\widehat{P_\mc A \varphi}(\xi) \equiv \tp{Proj}_{\ker \mc A(\xi)} \widehat\varphi(\xi).$
In \cite{GuerraRaita2020}, the authors proved:

\begin{teo}\label{teo:Lpest}
Fix $1<p<\infty$. An operator $\mc A$ as in \eqref{eq:defa} has constant rank if and only if
$$\Vert \D^k(\varphi - P_{\mc A} \varphi) \Vert_{L^p(\R^n)}
\leq C_p \Vert \mc A \varphi \Vert_{L^p(\R^n)} \qquad
\tp{ for all } \varphi \in C^\infty_c(\R^n, \mb V).$$
\end{teo}

At the endpoint $p=1$ and $p=\infty$ the above result should be contrasted with \textsc{Ornstein}'s non-inequality, see \cite{deLeeuw1962,Faraco2020,Kirchheim2016, Ornstein1962}. Nonetheless, at these endpoints there are weaker inequalities that one can use, see \cite{Raita2018a} for $p=1$, building on \cite{VanSchaftingen2011}.
The constant rank condition also admits a functional-analytic interpretation, see the corollary in \cite{GuerraRaita2020}.

Another characterization of constant rank operators was given by the second author in  \cite{Raita2018}. This characterization will be particularly useful for our purposes and the proof is based on a result of \textsc{Decell} \cite{DecellJr.1965}.

\begin{teo}\label{teo:b}
An operator $\mc A$ as in \eqref{eq:defa} if and only if there is a linear homogeneous differential operator $\mc B$
	with constant coefficients such that
	\begin{equation}
	\label{eq:exact}	\tp{im}\, \mc B(\xi)=\tp{ker\,} \mc A(\xi )\hs \tp{for all } \xi \in \R^n\exc \{0\}.\end{equation}
Moreover, $\mc B$ has constant rank as well.
\end{teo}

We will write, for some $B_\a \in \tp{Lin}(\mb U, \mb V)$, 
\begin{equation}
\label{eq:defb}
\mc B u = \sum_{|\a|=k} B_\a\p^\a u, \hs u \colon \Omega \subseteq
\R^n\to \mb U;
\end{equation}
equivalently, there is $T\in \tp{Lin}(\odot^k(\R^n, \mb U),\mb V)$ such that we can write in jet notation
\begin{equation}
\mc B = T \circ \D^k.\label{eq:C}
\end{equation}
We would like to emphasize that the construction of $\mc B$ given in \cite{Raita2018} is computable and that in fact one can always take $\mb U=\mb V$. We will refer to the potential operator $\mc B$ simply as the \textbf{potential} and to $\mc A$ as the \textbf{annihilator}, although this terminology is not standard.

From now onwards we shall assume implicitly that (\ref{eq:assumption}) holds and, for the sake of concreteness, we give a few examples when this is the case.

\begin{ex}\label{ex:intro}
	\begin{enumerate}
		\item Unconstrainted fields: if $\mc A=0$ then
                  $\Lambda_\mc A=\mb V$ and $\mc A$-quasiconvexity is
                  just ordinary convexity.
                \item \label{it:curlfree} Irrotational fields: let
                  $v\colon \R^n\to \R^n$ be a vector field and let
                  $\mc A=\tp{curl}$, where
$(\tp{curl\,}v)_{i,j}=\partial_i v_j-\partial_j v_i,\, i,j=1,
\dots, n.$
                  It is standard that $\mc A$-free vector fields have
                  a potential over simply connected domains, i.e.\
                  they can be written as the gradient of some other
                  function. One can also consider other variants
                  $\widetilde {\mc A}$ of the curl, for instance by
                  applying the curl row-wise to $m\times n$ matrices,
                  or more generally to higher order tensors, so that $\widetilde{\mc A}$-free fields correspond to $k$-th order gradients; see \cite{Ball1981} or \cite{Fonseca1999} for details.
		\item\label{it:divfree} Solenoidal fields: the constraint $\mc A=\tp{div\,}$ appears, for instance, in Fluid Dynamics, where the velocity field of an incompressible fluid is divergence-free.
		
		\item Examples \ref{it:curlfree}, \ref{it:divfree} fall in the  framework of exterior derivatives of differential forms \cite{Robbin1987}. 
		
		\item\label{it:linelast} Linear elasticity: in this case one studies integrands which depend only on  the symmetric gradient $\mc E(u)\equiv \frac 1 2(\D u+(\D u)^T)$ of the displacement $u\colon \Omega\subset \R^n\to \R^n$. A sufficiently regular vector field $v\colon \Omega \to \R^{n\times n}_\tp{sym}$ is a symmetric gradient if and only if it is ($\tp{curl curl}$)-free, where 
		$$(\tp{curl\,curl\,} v)_{i,j,k,l}\equiv \p^2_{kl} v_{ij} + \p^2_{ij} v_{kl}- \p^2_{jk} v_{il} - \p^2_{il} v_{jk}$$    
		is the Saint--Venant compatibility operator.
		\item Coupling of constraints: by combining several 		admissible constraints one obtains a new operator. For instance, by coupling \ref{it:divfree} and \ref{it:curlfree}  we have the equations of Electrostatics:
		$$\tp{div}\,B=0, \hs \tp{curl}\, E=0.$$
		If furthermore we couple these equations with \ref{it:linelast} we have the system of piezoelectricity. See \cite{Milton2002} for more examples. 
	\end{enumerate}
\end{ex}

Two important examples where the constant rank assumption does not hold are the operator associated to separate convexity \cite{Muller1999b,Tartar1993}, $\mc A v=(\partial_iv_j)_{i\neq j}$ acting on $v\colon \R^n\rightarrow \R^n$, and the operator associated to the incompressible Euler equations \cite{DeLellis2009,DeLellis2010,Szekelyhidi2012}.

\subsection{Cocanceling operators}

In order to discuss further properties of constant rank operators it will be convenient to employ simple algebraic properties of cocanceling operators, which for the reader's convenience we prove in this section.

\begin{ndef}\label{def:cocan}
	The operator $\mc B$ is said to be \textbf{cocanceling} if $\mb I_\mc B\equiv \bigcap_{\xi \in \mb S^{n-1}} \ker \mc B(\xi)=\{0\}$.
\end{ndef}  

This notion was introduced by \textsc{Van Schaftingen} in \cite{VanSchaftingen2011} and is equivalent to a critical linear $L^1$-estimate for $\mc B$-free fields. 
Typical examples of cocanceling operators are the divergence, the exterior derivative and the Saint--Venant compatibility operator, c.f.\ Example \ref{ex:intro}\ref{it:linelast}.

We recall a fundamental characterization of cocanceling operators \cite[Prop.~2.1]{VanSchaftingen2011}:
\begin{lema}\label{lema:char_cocanc}
	The following are equivalent:
	\begin{enumerate}
		\item $\mc A$ is cocanceling;
		\item $\int v=0$ for all $v \in C^\infty_c(\Omega, \mb V)$ such that $\mc A v =0$;
		\item If $v_0\in \mb V$ such that $\mc A\left(\delta_0v_0\right)=0$, then $v_0=0$.
	\end{enumerate}
\end{lema}

For our purposes, the relevance of cocancellation stems from the following simple result:

\begin{lema}\label{lema:cocan}
	Let $\mc B$ be as in (\ref{eq:defb}) and let $\mb J$ be a subspace which is such that $\mb U=\mb I_\mc B \oplus \mb J$. Then there is a choice of coordinates of $\mb U$ such that $\mc B$ can be represented as a block matrix
	$$\mc B =\begin{bmatrix}
	0_{\tp{Lin}(\mb I_\mc B, \mb V)} & \widetilde{\mc B}
	\end{bmatrix}$$
	where $\widetilde {\mc B}(\xi)\colon \mb J \to \mb V$ is cocanceling.
\end{lema}
An immediate consequence of Lemma \ref{lema:cocan} is that the space of 
$\mc B$-free fields contains $C^\infty_c(\R^n,\mb I_{\mc B})$. This space
is trivial if and only if $\mc B$ is cocanceling.
\begin{proof}
	The proof relies on \cite[Proposition~2.5]{VanSchaftingen2011}. Using the notation in \eqref{eq:defb}, we first claim that
	$$
	\mb I_{\mc B}=\bigcap_{|\alpha|=k}\ker B_\alpha.
	$$
	On one hand, if $B_\alpha v_0=0$ for all $\alpha$, then $\mc B(\xi)v_0=0$ for all $\xi\in\R^n$, so that $v_0\in\mb I_{\mc B}$. On the other hand, if $\sum_{|\alpha|=k}\xi^\alpha B_\alpha v_0=0$ for all $\xi\in\R^n$, by identifying coefficients, we obtain that $B_\alpha v_0=0$ for all $\alpha$.
	
	We  choose a basis of $\mb U$ such that the matrices $B_\alpha$ can be written as $B_\alpha=[0_{\tp{Lin}(\mb I_{\mc B},\mb V)}\,\,\widetilde B_\alpha]$ and define $\tilde{\mc B}(\xi)=\sum_{|\alpha|=k}\xi^\alpha \widetilde B_\alpha$. It is then clear that $\bigcap_{|\alpha|=k}\ker \widetilde B_\alpha=\{0\}$, which implies that $\widetilde {\mc B}$ is cocanceling.
\end{proof}
These results suggest that one can reduce statements about
non-{co}canceling operators to statements about cocanceling
operators, as often Lemma~\ref{lema:cocan} can be used to perform reductions. As a side note, we also record the following consequence:
\begin{cor}
	With the notation of Lemma~\ref{lema:cocan}, we have that $\Lambda_{\mc B}=\mb I_{\mc B}\times \Lambda_{\widetilde{\mc B}}$.
\end{cor}

\subsection{Further properties of potentials}

We shall now consider the following question: is there any meaningful sense in which the potential $\mc B$ associated with the operator $\mc A$ is unique? 
To find a canonical potential $\mc B$, one must take into account the following:
\begin{enumerate}
	\item $\mc B$ should have minimal order (for instance, if $\mc B$ is a potential, so is $|\xi|^2 \mc B(\xi)$);
	\item $\mc B$ is at best unique only modulo isomorphisms: if $Q\in \tp{GL}(\mb U)$, then $\mc B Q$ is another potential;
	\item $\mc B$ should be cocanceling, since adding columns of zeroes does not change $\tp{im}\, \mc B$ and hence preserves the exactness \eqref{eq:exact}, see Lemma \ref{lema:cocan}.
\end{enumerate}
While for many of the operators that occur in applications these conditions seem to suffice to single out a canonical potential (modulo isomorphisms of $\mb U$), in general they are not enough:

\begin{prop}\label{prop:nonunique}
	There is a first order constant rank operator $\mc A$ which admits two cocanceling potentials $\mc B_1, \mc B_2$ of minimal order which moreover satisfy $\mc B_1\neq \mc B_2 Q$ for all $Q\in \tp{Lin}(\mb U, \mb U)$.
\end{prop}

The proof of the proposition proceeds by construction of an explicit example; we relegate this to the appendix due to the long computations it requires. The example in the appendix is also one where it is  not possible to choose $\mc B$ to have the order of $\mc A$. It seems to have been known for quite some time that this is generically the case, see for instance \cite[page 445]{Kohn1965}.
A simpler example with this property can be found by considering the symmetric gradient of maps $u\colon \R^2\to \R^2$, which only has annihilators of order two or higher, see also \cite[Remarque 4]{Murat1978}. On the other hand, there is an example \cite[Example~3.10(d)]{Fonseca1999} of a first order annihilator for which the only known potential is $\D^k$. To sum up, we remark that one cannot make any assumption on the relation between the orders of $\mc A$ and $\mc B$.

From our perspective, Proposition \ref{prop:nonunique} implies that,
in the general, the operator $\mc B$ associated to the constant rank operator $\mc A$ has no physical content and is instead a useful mathematical tool. The potential is simply a polynomial parametrization of the wave cone; the physically relevant object is $\tp{ker\,}\mc A(\xi)$. This is already apparent in the Hilbert space axiomatization of Milton \cite{Milton1990} for composite materials, where the author postulates an orthogonal decomposition of the form
$$\mb V = \mc E_\xi \oplus \mc J_\xi, \hs \xi \neq 0;$$ the subspaces $\mc E_\xi$ and $\mc J_\xi$ correspond to the constraints satisfied by the applied  and induced fields, respectively---these would be, for instance, the electric field and current in the case of conductivity, hence the choice of notation. In practice, these constraints come from a partial differential equation and we have $\mc E_\xi=\ker \mc A(\xi)$ and  $\mc J_\xi = \ker \mc B^*(\xi)$ for some suitable operators.

\subsection{Function spaces}\label{sec:funcspaces}

In this subsection we gather some notation for function spaces associated with linear operators and prove some basic properties of these spaces. 
For our purposes it will be important to consider the space of $\mc A$-free test fields, i.e.\
$$C^\infty_{c,\mc A}(\Omega)\equiv \{v \in C^\infty_c(\Omega, \mb V): \mc A v =0\}.$$
In the general case where $\mc A$ is cocanceling (but does not necessarily have constant rank) it is unclear whether this space contains non-zero functions, while it always does in the non-cocanceling case as per Lemma \ref{lema:cocan}. Related to this  we have the following simple lemma (see also 
\cite[Proposition~2.1]{VanSchaftingen2011}):

\begin{lema}
	The space $C^\infty_{c,\mc A}(\R^n)$ is contained  in $\mathscr H^1(\R^n)$ if and only if $\mc A$ is cocanceling.
\end{lema}

\begin{proof}
	Suppose that $C^\infty_{c,\mc A}(\R^n)$ is contained in the Hardy space; since functions in $\mathscr H^1(\R^n)$ have zero mean then so do functions in $C^\infty_{c,\mc A}(\R^n)$ and this happens if and only if $\mc A$ is cocanceling. Moreover, test functions with  zero mean are contained in $\mathscr H^1(\R^n)$---in fact, they are dense there---and this proves the other direction.
\end{proof}

For $1\leq p\leq \infty$, we have the $L^p$-type spaces
\begin{align*}
L^p_{\mc A}(\Omega)\equiv\{
v \in L^p(\Omega,\mb V): \mc A v =0\}.
\end{align*}
Associated with $\mc B$, we define the $\mc B$-Sobolev-type spaces 
\begin{align}
  \label{eq:Bsobspace}
\mathscr W^{\mc B,p}(\Omega)\equiv
\tp{clos}_{u\mapsto\|\mc B u\|_{p}} C^\infty_c(\Omega, \mb U).
\end{align}
General properties of the $W^{\mc B,p}$-spaces can be found in the
recent works \cite{Breit2019,Gmeineder2019}.

When $\mc A$ is a constant rank operator and $1<p<\infty$ we have that
$C^\infty_{c,\mc A}$ is dense in $L^p_\mc A$; it is unclear whether
this holds for non constant rank operators. In fact, we have:
\begin{prop}\label{prop:Lp_A=BWB,p}
	If $\mc B$ is a potential for $\mc A$, we have
	\begin{equation}
	\mc B(\mathscr W^{\mc B,p}(\R^n))=\mc B(\dot W^{k,p}(\R^n,\mb U)) = L^p_\mc A(\R^n),
	\label{eq:spaces} 
      \end{equation}
      where $\dot W^{k,p}(\R^n,\mb U)$ denotes the usual homogeneous Sobolev space.
\end{prop}
Proposition \ref{prop:Lp_A=BWB,p} follows from the following Helmholtz--Hodge decomposition:

\begin{prop}\label{prop:hodge}
	Let $1<p<\infty$. A vector field $v\in L^p(\R^n, \mb V)$ can be uniquely\footnote{Here we do not mean that $u,w$ are  unique, but rather their images $\mc B u, \mc A^* w$.} decomposed as
	$$v=\mc B u +\mc A^* w $$
	for some $u\in \mathscr W^{\mc B,p}(\R^n)$,
	$w\in \mathscr W^{\mc A^*,p}(\R^n) $. 
	Moreover, this decomposition is continuous:
	$$\Vert \mc B u\Vert_{L^p(\R^n)} \leq C  \Vert v \Vert_{L^p(\R^n)}, \hs
	\Vert \mc A^*w \Vert_{L^p(\R^n)} \leq C \Vert \mc A v \Vert_{\dot{W}{^{-l,p}}(\R^n)}.
	$$
\end{prop}

Proposition \ref{prop:hodge} follows by standard methods from Theorem \ref{teo:b}, see for instance \cite{Fonseca1999,Giannetti2000}. We will  in fact construct $u\in \dot{W}{^{k,p}}(\R^n,\mb U),\,w\in \dot{W}{^{l,p}}(\R^n,\mb W)$.

\begin{proof}
	
	We begin by remarking that, once we have the decomposition, uniqueness follows straightforwardly from orthogonality. Indeed, consider a decomposition of zero, $0=\mc B u + \mc A^* w$. If $p'$ denotes the H\"older conjugate of $p$, let $\varphi \in L^{p'}(\R^n, \mb V)$ be arbitrary and write $\varphi = \mc B \chi + \mc A^*\psi$ for $\chi \in \mathscr{W}^{\mc B,p'}, \psi \in \mathscr{W}^{\mc A^*, p'}$. Then
	$$\int_{\R^n} \langle \mc B u, \varphi\rangle = 
	\int_{\R^n} \langle \mc B u, \mc B \chi \rangle+
	\int_{\R^n} \langle \mc B u, \mc A^* \psi\rangle=
	\int_{\R^n} \langle \mc B u, \mc B \chi \rangle =
	- \int_{\R^n} \langle \mc A^* w, \mc B \chi \rangle
	=0$$
	where we used twice the fact that $\int \langle \mc B b, \mc A^* a \rangle =0$ for all $b\in \mathscr W^{\mc B, p}, a \in \mathscr W^{\mc A^*,p'}$ in view of (\ref{eq:exact}).
	
	We assume that $\tp{ord}\,\mc B =k\geq l=\tp{ord}\,\mc A$, for otherwise we can replace $\mc B$ by $|\xi|^{2m} \mc B(\xi)$ for $m$ sufficiently large. Let $j=k-l$ and consider the homogeneous $k$-th order operator
	$$\Box\equiv \mc B \mc B^* + \mc A^*\mc A \t^{j};$$
	by the exactness relation (\ref{eq:exact}), this operator is elliptic, meaning that $\square(\xi)\in \tp{GL}(\mb V)$ for all $0\neq \xi\in\R^n$. This can be seen by letting $v_0\in\ker \square(\xi)$ and writing
	$$
	0=\langle \square(\xi)v_0,v_0 \rangle=\langle \mc B(\xi)\mc B^*(\xi)v_0,v_0 \rangle+|\xi|^{2j}\langle \mc A^*(\xi)\mc A(\xi)v_0,v_0 \rangle
=|\mc B^*(\xi)v_0|^2+|\xi|^2|\mc A(\xi)v_0|^2,	$$
	so $v_0\in\ker\mc B^*(\xi)\cap\ker \mc A(\xi)=\ker\mc B^*(\xi)\cap\mathrm{im\,} \mc B(\xi)=\left(\mathrm{im\,} \mc B(\xi)\right)^\perp\cap\mathrm{im\,} \mc B(\xi)=\{0\}$.
	
	Consequently, we can solve
	$\Box \varphi = v$
	for $\varphi \in \dot{W}{^{2k,p}}(\R^n,\mb V)$ with the elliptic estimate
	\begin{equation}
	\Vert D^{2k} \varphi\Vert_{L^p(\R^n)} \leq C \Vert v \Vert_{L^p(\R^n)}.
	\label{eq:ellipticest}
	\end{equation}
	Now define
	$$u\equiv \mc B^* \varphi, \hs w \equiv \mc A \t^j \varphi;$$
	then (\ref{eq:ellipticest}) already gives the estimate for $\mc B u$ in the statement, as well as a similar estimate for $\mc A^*w$. Note that due to the bounds in \eqref{eq:ellipticest}, we can assume that $\varphi\in C^\infty_c(\R^n,\mb V)$, otherwise it can be replaced with an approximating sequence $\varphi_j$ such that $\square\varphi_j$ converges to $v$ in $L^p$.
	
	To get the better estimate for $\mc A^*w$, we apply $\mc A$ to the decomposition to get $\mc A v=\mc A\mc A^*w$, so that we can compute in Fourier space, for $\xi\neq 0$,
	$$
	\mc A^*(\xi)\hat{w}(\xi)=\mc A^\dagger(\xi)\mc A(\xi)\mc A^*(\xi) \hat w(\xi)=\mc A^\dagger\left(\frac{\xi}{|\xi|}\right)\frac{\widehat{\mc A v}(\xi)}{|\xi|^l},
	$$
	where we used the fact that $A^\dagger A=\mathrm{Proj}_{\mathrm{im\,}A^*}$. The H\"ormander--Mihlin multiplier theorem then implies that
	$$
	\|\mc A^* w\|_{L^p(\R^n)}\leq C\left\|\mc F^{-1}\left(\frac{\widehat{\mc A v}(\xi)}{|\xi|^l}\right)\right\|_{L^p(\R^n)}=C\|\mc Av\|_{\dot{W}{^{-l,p}}(\R^n)},
	$$
	which concludes the proof.
\end{proof}
\begin{proof}[Proof of Proposition~\ref{prop:Lp_A=BWB,p}]
	Let $v\in L^p(\R^n)$ with $\mc Av=0$. Using Proposition~\ref{prop:hodge}, we have that $v=\mc B u+f$, where $f=\mc A^*w\in L^p(\R^n,\mb V)$ is such that $\mc B^* f=0$. This follows since the exactness (\ref{eq:exact}) can equivalently be written as $\tp{im\,}\mc A^*(\xi)=\ker\mc B^*(\xi)$ for $\xi\neq 0$, hence $\mc B^*\circ\mc A^*=0$. On the other hand, since $v$ is $\mc A$-free, we also obtain $\mc A f=0$. Therefore $\square f=0$, so that $f$ is analytic by the ellipticity of $\square$. Since $f\in L^p(\R^n)$, we conclude that $f=0$, which implies the only non-trivial inclusion in \eqref{eq:spaces}.
\end{proof}
Through the multiplier $\mc A^\dagger({\xi}/{|\xi|})$, the proof of the Helmholtz--Hodge
decomposition in Proposition~\ref{prop:hodge} relies heavily on the Calder\'on--Zygmund theory to solve an auxiliary partial differential equation in full space. Having a similar decomposition that holds in bounded domains may be a viable tool to tackle other problems in the field. This motivates the following:
\begin{question}\label{qu:helmholtz}
	Let $\Omega\subset\R^n$ be a sufficiently regular bounded domain and $1<p<\infty$. Is it the case that each $v\in L^p(\Omega,\mb V)$ has a unique decomposition
	$$
	v=\mc B u+\mc A^* w+h,
	$$
	where $u\in \mathscr{W}{^{\mc B,p}}(\Omega)$, $w\in\mathscr{W}{^{\mc A^*,p}}(\Omega)$, and $\mc B^* h=0$, $\mc A h=0$ in the sense of distributions, with the bounds
	$$
	\|\mc B u\|_{L^p(\Omega)}+ \|h\|_{L^p(\Omega)} \leq C_p \|v\|_{L^p(\Omega)},\quad \|\mc A^*w \|_{L^p(\Omega)}\leq C\|\mc Av\|_{\dot{W}{^{-l,p}(\Omega)}}?
	$$
\end{question}
It is known that the domain $\Omega$ cannot be taken to be an arbitrary open set \cite{Hajlasz1996}.
The ``harmonic'' field $h$ is analytic in $\Omega$, since it satisfies $\square h=0$. 
It is also known that one cannot hope for a decomposition with $h=0$, since this is not the case for
exterior differentials and codifferentials; in this situation,
furthermore, the answer to the question is positive, see for instance
\cite{Sil2017} for an elementary proof. Question \ref{qu:helmholtz} is also true for $p=2$:
\begin{proof}[Answer to Question~\ref{qu:helmholtz} for $p=2$]
	Note that the orthogonal
	complement in $L^2(\Omega,\mb V)$ of $X\equiv \{\mc B u: u \in C^\infty_c(\Omega, \mb U)\}$ is
	$$
	Y\equiv\{v\in L^2(\Omega,\mb V): \mc B^* v = 0 \tp{ in the sense of distributions}\}.
	$$
	This follows from the following identity, which holds for all $u\in C_c^\infty(\Omega,\mb U)$ and $f\in L^2(\Omega,\mb V)$:
	$$
	\langle f,\mc B u\rangle_{L^2}=\int_{\Omega}\langle f,\mc B u\rangle_{\mb V}\dif x=\langle f,\mc B u\rangle_{\mathscr D^\prime,\mathscr D}=(-1)^k\langle \mc B^*f, u\rangle_{\mathscr D^\prime,\mathscr D}.
	$$
	The projection theorem yields the orthogonal decomposition $L^2(\Omega,\mb V)=\overline{X}\oplus Y$. We then note that $Z\equiv \{\mc A^* w\colon w\in C^\infty_c(\Omega,\mb W)\}$ is a subspace of $Y$. An analogous argument shows that the orthogonal complement of $Z$ in $Y$ is $H\equiv \{h\in L^2(\Omega,\mb V)\colon \mc A h=0,\,\mc B^*h=0\}$. In particular, we obtain the orthogonal decomposition $L^2(\Omega,\mb V)=\overline{X}\oplus\overline{Z}\oplus H$, which gives the claim, except for the negative Sobolev bound. To prove this as well, note that we already have a sequence $w_j\in C^\infty_c(\Omega,\mb W)$ such that $\mc A^* w_j\to v-\mc B u-h$ in $L^2(\Omega,\mb V)$, so that $\mc A\mc A^* w_j\to\mc A v$ in $\dot{W}{^{-l,2}(\Omega,\mb V)}$. It remains to recall the last estimate from the proof of Proposition~\ref{prop:hodge}, i.e.
	$$
	\|\mc A^* w_j\|_{L^2(\Omega)}=\|\mc A^* w_j\|_{L^2(\R^n)}\leq C\|\mc A\mc A^* w_j\|_{\dot{W}{^{-l,2}}(\R^n)}=C\|\mc A\mc A^* w_j\|_{\dot{W}{^{-l,2}}(\Omega)},
	$$
	where the equalities  follow since $w_j$ are supported inside $\Omega$.
\end{proof}

\section{$\mc A$-quasiconvexity and weak lower semicontinuity}\label{sec:lsc}

We recall the following definition \cite{Fonseca1999}, generalizing the previous notion of \textsc{Morrey} \cite{Morrey1952}:

\begin{ndef}\label{def:A-qc}
	A locally bounded, Borel function $F\colon \mb V \to \R$ is \textbf{$\mc A$-quasiconvex} if 
	$$0\leq \int_{[0,1]^n} F(z+ v(x))-F(z) \d x $$
	for all $z\in \mb V$ and all $v\in C^\infty_\tp{per}([0,1]^n,
        \mb V)$ such that $\mc A v=0$ and $\int_{[0,1]^n} v =0$.

Moreover, $F\colon \mb V \to \R$ is said to be $\mc A$-\textbf{quasiaffine} if both $F$ and $-F$ are $\mc A$-quasiconvex.
\end{ndef}

An important consequence of Theorem \ref{teo:b} is that, under a constant rank assumption, the above definition can be changed to resemble more closely the original definition of quasiconvexity in the gradient case (see \cite[Corollary 1]{Raita2018}):

\begin{cor}\label{cor:Bqc}Let $\Omega\subseteq \R^n$ be a non-empty
  open subset, $\mc A$ be a constant rank operator as in the setup of
        Theorem \ref{teo:b} and let $\mc B$ be an operator as in
        (\ref{eq:defb}) which satisfies (\ref{eq:exact}).
	A locally bounded Borel function $F\colon \mb V \to \R$ is
        $\mc A$-quasiconvex, respectively $\mc A$-quasiaffine, if and only if
	$$0\leq \int_\Omega F(z+\mc B u(y))-F(z) \d y,$$
        respectively
        \begin{equation}
\label{eq:B-qa}
0=\int_\Omega F(z + \mc B u) - F(z) \d x, \end{equation}
	for all $z \in \mb V$ and all $u \in C^\infty_c(\Omega, \mb U)$.
\end{cor}

In particular, Corollary \ref{cor:Bqc} shows that $F\colon \mb V\to
\R$ is $\mc A$-quasiaffine if and only if
 for all $z\in \mb V$ and $u \in C^\infty_c(\Omega, \mb U)$ and every non-empty open set $\Omega \subset \R^n$.

Besides constant rank, it will be important to assume that the wave cone of $\mc A$ spans the entire space. This is related to the following well-known lemma \cite[Section~2.5]{Arroyo-Rabasa2018};  we give a proof only for the sake of completeness.

\begin{lema}\label{lema:ctsAqc}
	We have $\tp{span}\,\Lambda_\mc A=\mb V$ if and only if all $\mc
	A$-quasiconvex functions are continuous.
\end{lema}

\begin{proof}
	The direction $\Rightarrow$ is standard and follows from the fact that
	any such function is $\Lambda$-convex and $\Lambda$-convex functions
	are (locally Lipschitz) continuous in $\tp{span}\, \Lambda$, see
	e.g. \cite[Lemma 2.3]{Kirchheim2016}.
	To prove $\Leftarrow$ assume $\tp{span}\, \Lambda \neq \mb V$. Then
	we can write $(v_1, v_2)\in\mb V=\tp{span}\,\Lambda\oplus \mb
	{\widetilde V}$ where
	$\mb{\widetilde V}\neq \{0\}$. The function defined by
	$F(v_1,v_2)=1_{\{v_2=0\}}(v_1, v_2)$ is a discontinuous $\mc
	A$-quasiconvex function; in fact, it is even $\mc A$-quasiaffine. Here we used the fact that periodic $\mc A$-free fields take their values in $\tp{span}\,\Lambda$.
\end{proof}

In what follows we will make the standard assumption that $F\colon \mb V\to \R$ satisfies a $p$-growth condition
\begin{equation}
\label{eq:Gp}
|F(v)|\leq C(|v|^p+1)
\tag{$G_p$}
\end{equation}
The importance of $\mc A$-quasiconvexity is its relation to lower semicontinuity, made precise by the following fundamental result by \textsc{Fonseca}--\textsc{M\"uller} \cite{Fonseca1999} (see also \cite[Remark 1.3]{Arroyo-Rabasa2018}):

\begin{teo}\label{teo:lsc}
	Let $\mc A$ have constant rank and let $F\colon \Omega \times \mb V\to \R$ be a Carath\'eodory integrand. The functional $v\mapsto \int_\Omega F(x,v(x)) \d x$ is sequentially weakly-$*$ lower semicontinuous on $L_{\mc A}^{ \infty}(\Omega)$ if and only if for each fixed $x_0 \in \Omega$ the map $F(x_0, \cdot)$ is $\mc A$-quasiconvex.
	
	Moreover, if (\ref{eq:Gp}) holds for some $1<p<\infty$ and we fix $1<p<q$, then we have
	$$\begin{rcases}v_j \w v &\tp{in } L^{q}(\Omega)\\
	\mc A v_j \to 0 & \tp{in } W^{-l,q}(\Omega)
	\end{rcases}
	\hs \implies \hs
	\liminf_{j\to \infty} \int_\Omega F(x,v_j(x)) \d x \geq \int_\Omega F(x,v(x))\d x $$
	if and only if for a.e.\ $x_0\in \Omega$ the map $F(x_0,\cdot)$ is $\mc A$-quasiconvex.
\end{teo}

We remark that, in general, the conclusion of the theorem is false in the critical case $p=q$ unless one assumes additional structure on either the integrand, for instance positivity as done in \cite{Fonseca1999}, or on the sequence, for instance that it does not concentrate on the boundary. A counterexample illustrating this failure was given for $\mc A=\tp{curl}$ and $F=\det$ in \cite[Example 7.1, 7.3]{Ball1984}. We refer the reader to  \cite{Benesova2017} for a detailed discussion of this issue.

The following lemma is well-known and was proved in the $\mc A=\tp{curl}$  case in \cite{Acerbi1984,Marcellini1985}. 

\begin{lema}\label{lema:loclipest}
	Assume $\Lambda$ spans $\mb V$. If $F\colon \mb V \to \R$ is $\Lambda$-convex and satisfies (\ref{eq:Gp}) then
	$$|F(v)-F(w)|\leq C(1+ |v|^{p-1} + |w|^{p-1}) |v-w|$$
	for all $v, w \in \R^d$.
\end{lema}

\begin{proof}
	By the spanning condition, $F$ is Lipschitz and, for $v, w \in B_r(0)\subset \mb V$,
	$$|F(v)-F(w)|\leq \frac{C}{r}\,\tp{osc}(F, B_{2r})
	|v-w|,$$
	where $C$ depends only on $\Lambda$; see \cite[Lemma 2.3]{Kirchheim2016}. Using (\ref{eq:Gp}) and the triangle inequality, we get
	$$|F(v)-F(w)|\leq C \left(1+\frac{|v|^p}{r}+\frac{|w|^p}{r} \right)|v-w|\leq C(1+ |v|^{p-1}+|w|^{p-1})|v-w|$$
	where we also assumed without loss of generality that $r\geq 1$.
\end{proof}

We are now ready to begin the proof of the main result of this section. Recall that we always assume (\ref{eq:assumption}). The next proposition, although relatively simple, is a crucial ingredient in the proof of Theorem \ref{teo:ourlsc} below. The point is that when a weakly convergent sequence does not concentrate on the boundary it can be replaced by a sequence of potentials.

\begin{prop}\label{prop:bogdan11}
	Let $\Omega$ be a bounded domain. 
	Let $v_j, v \in L^p(\Omega, \mb V)$ be such that
	$$v_j \w v \tp{ in }L^p(\Omega,\mb V),\hs 
	\mc A v_j \to \mc A v \tp{ in } W^{-l, p}_\tp{loc}(\Omega, \mb V)
	$$
	and moreover let $\lambda$ be such that
	$|v_j|^p\wstar \lambda$ in $\mc M(\overline \Omega)$.
	Assume that $\lambda(\p \Omega)=0$.
	Up to passing to subsequences in $(v_j)$, there is a sequence $u_j \in C^\infty_c(\Omega, \mb U)$ such that 
	$$v_j-v-\mc B u_j \to 0\tp{ in } L^p(\Omega,\mb V). $$
\end{prop}

\begin{proof}
	By linearity we may assume that $v=0$. Let $U\Subset V \Subset \Omega$ to be determined later and take $\eta \in C^\infty_c(\Omega)$ with $1_{U}\leq \eta\leq 1_{V}$ and $|\D^m \eta|\leq 2d^{-m}$ for $m=1, \dots, k$; here
	$d\equiv \tp{dist}(U, \p V)$. Write, using the Helmholtz-Hodge decomposition of Proposition \ref{prop:hodge}, 
	$$\widetilde v_j \equiv  \eta v_j, \hs \widetilde v_j = \mc B u_j + w_j,$$
	where we have extended $\widetilde v_j$ by zero outside $\Omega$ so that it is in $L^p(\R^n, \mb V)$. Moreover, we have 
	$$\Vert v_j - \mc B u_j \Vert_{L^p(\Omega)} \leq \Vert v_j - \widetilde v_j \Vert_{L^p(\Omega)} + \Vert \widetilde v_j - \mc B u_j \Vert_{L^p(\Omega)} \lesssim  \Vert v_j - \widetilde v_j \Vert_{L^p(\Omega)}+\Vert \mc A \widetilde v_j \Vert_{W^{-l,p}(\Omega)}.$$
	Let us estimate the first term: since $\lambda$ is a positive measure,
	$$\lim_{j\to \infty} \Vert (1-\eta) v_j \Vert_{L^p(\Omega)} = \int_{\overline\Omega} (1-\eta)^p \d \lambda\leq \lambda(\overline \Omega\exc U).$$
	Taking $U\uparrow \Omega$ the left-hand side goes to zero by the dominated convergence theorem, since $\lambda(\p \Omega)=0$. 
	For the second term, we have
	\begin{align*}
	\Vert \mc A(\eta v_j)\Vert_{W^{-l,p}(\Omega)} &\leq \Vert
	\eta \mc A v_j \Vert_{W^{-l,p}(V)} + \sum_{i=1}^k
	\Vert B_i[\D^i \eta, \D^{k-i} v_j]\Vert_{W^{-l,p}(V)}
	\end{align*}
	where the $B_i$ are fixed bilinear pairings given by the chain rule.
	For the first term note that, up to taking subsequences in $v_j$ if necessary, we can assume that
	$$\Vert
	\eta \mc A v_j \Vert_{W^{-l,p}(V)} \leq \frac 1 j$$
	by our hypothesis.
	The second term can be bounded by
	$$
	\Vert B_i[\D^i \eta, \D^{k-i} v_j]\Vert_{W^{-l,p}(V)}\lesssim \frac{\Vert  \D^{k-i}v_j\Vert_{W^{-l,p}(V)}}{d^i}\lesssim 
	\frac{\Vert v_j\Vert_{L^p(V)}}{d^i}.
	$$
	Thus, picking $U,V\uparrow \Omega$ such that $d$ approaches zero sufficiently slowly, this term also goes to zero. This finishes the proof: although $u_j$ is only in $\mathscr W^{\mc B,p}(\Omega)$, by definition of this Sobolev space there are $\widetilde u_j \in C^\infty_c(\Omega, \mb U)$ with $\Vert \mc B(u_j-\widetilde u_j)\Vert_p \to 0$.
\end{proof}

We proceed to the proof of the main result of this section; it is inspired by standard lower semicontinuity proofs in the gradient case \cite{Acerbi1984,Chen2017,Kristensen1999b,Marcellini1985,Meyers1965,Morrey1952}.

\begin{teo}\label{teo:ourlsc}
	Let $\Omega\subset \R^n$ be a bounded domain.
	If $F\colon \mb V\to \R$ is $\mc A$-quasiconvex and satisfies 
	(\ref{eq:Gp}) then, whenever 
	$$v_j \w v \tp{ in }L^p(\Omega,\mb V),\hs 
	\mc A v_j \to \mc A v \tp{ in } W^{-l, p}_\tp{loc}(\Omega, \mb V),
	$$ for all $\rho \in C_c^\infty(\Omega)$ with $\rho\geq 0$ we have
	$$\liminf_{j\to \infty} \int_\Omega	 \rho F(v_j) \d x \geq \int_\Omega \rho F(v) \d x.$$ 
\end{teo}

\begin{proof}
	By taking a subsequence, we can assume that 
	$|v_j|^p \wstar \lambda$ in $\mc M(\Omega)$. Let us also fix $\rho \in C^\infty_c(\Omega)$ with $\rho\geq 0$ and $\e\in(0,1)$.
	
	\textbf{Step 1:} We can find $\widetilde v\in C^\infty_c(\Omega, \mb V)$ such that $\Vert v - \widetilde v \Vert_{p}<\e$. Let us also take $\delta\in(0,1)$ such that, given any 
	triangulation $\widetilde {\mc T}$ of $\R^n$ with $\sup_{T\in \widetilde {\mc T}} \tp{diam} \,T <\delta$, we can find a function $a$, constant in each $T\in \widetilde{\mc T}$, with the bound 
	$\Vert \widetilde v - a\Vert_{L^p(\Omega)}<\e$. In particular, $a$ satisfies
	\begin{equation}
	\label{eq:a}
	\Vert a\Vert_{L^p(\Omega)} \leq 2 \e + \Vert v\Vert_{L^p(\Omega)} <2+\Vert v \Vert_{L^p(\Omega)}.
	\end{equation}
	
	We need to wiggle the triangulation sightly so that Proposition \ref{prop:bogdan11} becomes applicable. For this, let $\mc T_\Omega\equiv \{T \in  \widetilde{\mc T}: T \cap B_2(\Omega)\neq \vazio\}$.
	Take a direction $e\in \mb S^{n-1}$ which is not tangent to any face of any simplex $T \in \mc  T_\Omega$. Then, given a face $\sigma$ of $T$,  the sets $t e + \sigma$, for $t\in (0,\delta)$, are disjoint. This shows that the set
	$$\{t \in (0,\delta): \lambda(t e +\sigma)>0\}$$
	is at most countable and hence so is the set
	$$E\equiv \bigcup_{T\in \mc T_\Omega} \{t \in (0,\delta): \lambda(t e + \p T)>0\}.$$
	Select $t\in (0,\delta)\exc E$ and define the final triangulation $\mc T\equiv t e + \mc T_\Omega$, which contains $B_1(\Omega)$. Choose $a$ to be constant in each $T\in \mc T$ and satisfy (\ref{eq:a}).
	
	\textbf{Step 2:}
	Let us write $w_j \equiv a+ v_j -v \in L^p(\Omega,\mb V)$; then
	\begin{align*}
	\int_\Omega \rho(F(v_j)-F(v))\d x = 
	\int_\Omega \rho (F(v_j)-&F(w_j)) \d x + \int_\Omega \rho (F(w_j)-F(a)) \d x \\ & + \int_\Omega \rho(F(a)-F(v)) \d x \equiv \tp{I} + \tp{II} + \tp{III}.
	\end{align*}
	Using the local Lipschitz estimate of Lemma \ref{lema:loclipest}, we get
	\begin{align*}
	|\tp{I}+\tp{III}|&\lesssim \int_\Omega \rho (1+|v_j|^{p-1} + |w_j|^{p-1}) |v_j-w_j| \d x +
	\int_\Omega \rho (1+|v|^{p-1} + |a|^{p-1}) |v-a| \d x \\
	& \leq \max \rho \int_\Omega (1+ |v_j|^{p-1} + 2^p |v_j|^{p-1} + 2^p |v-a|^{p-1})|v-a| \d x  \\
	&  \hs +\max \rho \left(\int_\Omega (1+ |v|^{p-1} + |a|^{p-1})^\frac{p}{p-1} \d x \right)^\frac{p-1}{p} \left(\int_\Omega |v-a|^p\right)^\frac 1 p
	\end{align*}
	Thus, from (\ref{eq:a}) and using H\"older again for the first term, we find that
	$$|\tp{I}+\tp{III}|\leq C\left(1+\Vert v \Vert^{p-1}_p + \sup_j \Vert v_j\Vert_p^{p-1}\right)\e=O(\e)$$
	where $C$ now also depends on $\rho$. To summarize, we have $w_j\w a$ in $L^p(\Omega, \mb V)$ and we have shown that
	\begin{equation}
	\label{eq:step1}
	\liminf_{j\to \infty} \int_\Omega \rho (F(v_j) - F(v)) \d x = O(\e) + \liminf_{j\to \infty} \int_\Omega \rho (F(w_j) -F(a))\d x.
	\end{equation}
	
	\textbf{Step 3:} Since $\mc T$ triangulates $\Omega$ we have
	\begin{equation}
	\int_\Omega \rho(F(w_j)-F(a)) \d x = \sum_{T\in \mc T} \int_{T\cap \Omega} \rho(F(w_j)-F(a)) \d x.\label{eq:triang}
	\end{equation}
	Using Proposition \ref{prop:bogdan11},
	take for each $T\in \mc T$ a sequence $u_{j,T}\equiv u_{j}\in C^ \infty_c(T, \mb V)$ such that $w_j -a - \mc B u_{j}\to 0$ in $L^p(T,\mb V)$. By Lemma \ref{lema:loclipest}, 
	$$\int_T F(w_j) -F(a+\mc B  u_j)\d x \to 0$$
	and since $F$ is $\mc A$-quasiconvex, from Corollary \ref{cor:Bqc},
	$$\int_T F(a+\mc B u_j) - F(a) \d x \geq 0.$$
	Putting these together, we have shown that
	\begin{equation}
	\liminf_{j\to \infty} \int_T F(w_j)-F(a) \d x \geq 0.\label{eq:liminfw}
	\end{equation}
	Take for each $T \in \mc T$ a point $x_T\in \mc T$ and note that, from (\ref{eq:triang}),
	\begin{align*}
	&\int_\Omega \rho(F(w_j)-F(a)) \d x =\\
	&\hspace{0.2cm} =
	\sum_{T \in \mc T} \rho(x_T) \int_{T\cap \Omega} F(w_j)-F(a) \d x  + \int_{T\cap \Omega} (\rho-\rho(x_T))(F(w_j)-F(a)) \d x \\
	&\hspace{0.2cm} \geq \sum_{T\in \mc T} \rho(x_T) \int_{T\cap \Omega} F(w_j) -F(a)\d x - \max_{T\in \mc T} \tp{diam}\,\rho(T) \int_{\Omega} C (1+|w_j|^{p-1} +|a|^{p-1}) |w_j-a|\d x.
	\end{align*}
	To bound the first term we use (\ref{eq:liminfw}) and to bound the second we recall that $w_j-a=v_j-v$ and use the estimate (\ref{eq:a}) for $a$:
	$$\liminf_{j\to \infty} \int_\Omega \rho(F(w_j)-F(a)) \d x \geq
	- C\max_{T\in \mc T} \tp{diam}\,\rho(T) 
	\left[
	\int_{\Omega}  1+|v|^p \d x+\sup_j \int_{\Omega}  |v_j|^p \d x\right].$$
	Since $\rho$ has compact support it is uniformly continuous and since $\tp{diam}\, T<\delta$ for $T\in \mc T$ we have that $\max_{T\in \mc T} \tp{diam}\,\rho(T)\to 0 $ as $\delta\to 0$. Finally, using (\ref{eq:step1}) and sending $\e\to 0$ the conclusion follows.
\end{proof}

The above proof can be easily adapted to the case where we do not assume that $\rho$ has compact support, instead assuming that the negative part of the integrand has $q$-growth for $q<p$, see e.g. the proofs in \cite{Dacorogna2007,Marcellini1985}. This recovers the second case of Theorem \ref{teo:lsc} above.

\section{Null Lagrangians and weak continuity}\label{sec:nullLags}

We begin by recording the following definition:

\begin{ndef}
	Given a $C^1$ integrand $F\colon \mb V\to \R$, we say that it is an $\mc A$-\textbf{null Lagrangian} if it satisfies, in the sense of distributions,
	\begin{equation}
	\mc B^* \left(\D F(\mc B u)\right)=0, \label{eq:el}
	\end{equation}
	for all $u\in C^k(\overline \Omega, \mb U)$. When the choice of $\mc A$ is implicit from the context we refer to such integrands simply as null Lagrangians.
\end{ndef}

We remark that one can also consider null Lagrangians depending on lower order terms, as in \cite{Olver1988}, but we shall not pursue this here.

Having Theorem \ref{teo:ourlsc} at our disposal, we can give a first
abstract characterization of $\mc A$-quasiaffine maps under the main
assumption (\ref{eq:assumption}); this will be improved in the next
section and quantified in Section \ref{sec:estimates}.
The following proposition is modelled on  \cite[Theorem 3.4]{Ball1981}.

\begin{prop}\label{prop:abstractnl}
	Let $F\colon \mb V \to \R$ be locally bounded and Borel and let $\Omega$ be  a bounded domain. The following are equivalent:
	\begin{enumerate}
		\item \label{it:A-qa}$F$ is $\mc A$-quasiaffine;
		\item \label{it:Piola}$F$ is an $\mc A$-null Lagrangian;
		\item \label{it:weak*cts}$F\colon L^{\infty}_{\mc A}(\Omega)\to L^\infty(\Omega)$ is sequentially weakly-$*$ continuous;
		\item \label{it:improved} $F$ is a polynomial of degree $s\leq \min\{n,\dim \mb V\}$ and 
		$$\begin{rcases}
		v_j \w v & \tp{in } L^s(\Omega,\mb V)\\
		\mc A v_j \to \mc A v & \tp{in } W^{-l,s}_\tp{loc}(\Omega,\mb V)
		\end{rcases}\hs \implies \hs F(v_j)\wstar F(v) \tp{ in } \mathscr D'(\Omega).$$
		
	\end{enumerate}
\end{prop}

In light of \ref{it:Piola} above we will sometimes call $\mc A$-quasiaffine maps \textbf{null Lagrangians}, as it is usual in the Calculus of Variations literature.

\begin{proof}
	\ref{it:A-qa} $\Leftrightarrow$ \ref{it:weak*cts}: It is clear that \ref{it:weak*cts} holds if and only if, for any $\varphi \in C(\Omega)$,
	the functionals $u\mapsto \pm \int_\Omega \varphi(x) F(v(x)) \d x$ are sequentially weakly$^*$ lower semicontinuous on $ L^{\infty}_{\mc A}(\Omega)$. By Theorem \ref{teo:lsc} this happens if and only if $F$ is $\mc A$-quasiaffine.

	Clearly  
	\ref{it:improved} $\Rightarrow$ \ref{it:weak*cts}. We now prove
	\ref{it:weak*cts} $\Rightarrow$ \ref{it:improved}, by an argument similar to the one in the first paragraph. 
	It is well-known that $F$ must be $\Lambda$-affine (see e.g. \cite{Tartar1979}), i.e.\ it is affine along lines parallel to $\Lambda$. Since 
	$\tp{span}\,\Lambda=\mb V$, it must be a polynomial of degree $s\leq
	\dim \mb V$ and the inequality $s\leq n$ follows from (\ref{eq:derivatives}) below.
	We apply Theorem \ref{teo:ourlsc} to conclude that if the premise of the implication in \ref{it:improved} holds then $\int_\Omega \varphi(x)F(x,v_j(x))\dif x\rightarrow\int_\Omega \varphi(x)F(x,v(x))\dif x$.

	\ref{it:A-qa} $\Rightarrow$ \ref{it:Piola}: We already know that $F$ is a polynomial so in particular it is smooth. Let us take $u_n,\varphi \in C^\infty_c(\Omega, \R^b)$ and $t>0$. Then, by (\ref{eq:B-qa}),
	$$0=\left.\frac{\d }{\d t}\right|_{t=0} \int_\Omega F(\mc B u_n + t \mc B \varphi_n) \d x= \sum_{i=1}^d \int_\Omega \frac{\p F}{\p v^i}(\mc B u_n) (\mc B \varphi)^i \d x.$$
	Choosing $u_n\to u$ in $C^k(\supp \varphi)$, we obtain \ref{it:Piola}. The converse direction is identical. 
\end{proof}

Most of the above proposition is essentially contained in the literature, as becomes clear from the proof. The only novelty is \ref{it:improved}, which improves the integrability required for \textsc{Murat}'s result \cite{Murat1981} to hold: even in the simplest case where $\mc B=\D^k$, it only follows from his result that a polynomial of degree three is sequentially weakly continuous as a map 
$ W^{k,4}(\Omega)\to \mathscr D'(\Omega) $; this had already been observed
and improved in \cite{Ball1981}, see also \cite{Reshetnyak1967,Reshetnyak1968}, but here it is extended to an
arbitrary constant rank operator.

While Proposition \ref{prop:abstractnl} gives an abstract characterization of null Lagrangians it is relevant to have an effective way of computing them.
For an operator\footnote{Strictly speaking, in
	\cite{Murat1981,Tartar1979} it is assumed that $\mc A$ is a
	first-order operator, but one
	can easily check that the proof carries through to the case where
	$\mc A$ is a general homogeneous $l$-th order operator.} $\mc A$ not necessarily of constant rank
\textsc{Tartar}  \cite{Tartar1979} showed that \ref{it:weak*cts} implies the algebraic condition
\begin{equation}
\label{eq:derivatives}
\D^r F(v)[\lambda_1, \dots, \lambda_r]=0 \tp{ for all } (\lambda_1, \xi_1), \dots, (\lambda_r, \xi_r) \in V \tp{ with } \tp{rank\,}(\xi_1, \dots, \xi_r)<r
\end{equation}
for all $v\in \mb V$ and all $r\geq 2$. \textsc{Murat} \cite{Murat1981} then proved  that if moreover $\mc A$ has constant rank then these conditions are in fact sufficient, i.e.\ (\ref{eq:derivatives}) is equivalent to \ref{it:weak*cts}. Unfortunately, it is in general unclear what are the polynomials, if any, satisfying the above restriction. 
\textsc{Murat} \cite[page 93]{Murat1978} was already aware of this difficulty (emphasis not ours):

\begin{quote}
	Encore faut-il, dans chaque cas particulier, \textit{expliciter} quels sont les polyn\^omes homog\`enes de degr\'e $r$ qui satisfont [(\ref{eq:derivatives})]. Cela conduit \`a des calculs alg\'ebriques qui sont parfois difficiles, voire inextricables.
\end{quote}
Even in the case where $\mc B=\D^k$ it is by no means easy to find all the weakly continuous functions. The following result \cite[Theorem 4.1]{Ball1981} relies on deep algebraic facts:

\begin{teo}\label{teo:BCO}
	Let $F\colon \odot^k (\R^n,\R^m)\to \R$ be continuous. Then $F=F(\D^k u)$ is $\D^k$-quasiaffine if and only if it is an affine combination of Jacobians of $U\equiv \D^{k-1}u$, by which we mean that there exist constants $c_M\in\R$ such that
	$$
	F=F(0)+\sum_{M} c_M M(\D U) 
	$$
	where $M\colon \R^{N\times n}\rightarrow \R$ runs over all
        $s\times s$ minors of $N\times n$ matrices, for
        $N=\dim\odot^{k-1} (\R^n,\R^m) $ and $s=1,\ldots\min \{n,N\}$.
\end{teo}

It appears that this result was proved independently around the same time in \cite{Anderson1980}.
We are interested in using the above theorem to make the characterization of Proposition
\ref{prop:abstractnl} more explicit.
Let us write, following \cite[\S 4]{Kirchheim2016},
$$\mc D(n,k, \mb U)\equiv \left\{u\otimes \xi^{\otimes k}:
u\in \mb U, \xi \in \R^n \right\};$$
this cone spans $\odot^k(\R^n, \mb U)$ and when $k=1$ is the usual cone of rank-one linear transformations. 
Going back to (\ref{eq:C}), we  note that it implies that, for $v\in \mb V$,
$$\mc B(\xi)v=T(v \otimes \xi^{\otimes k}).$$
Since $\tp{im}\, \mc B(\xi)=\tp{ker\,} \mc A(\xi )$, it follows from the definition of $\Lambda$ that $T $ maps the cone $\mc D(n,k, \mb U)$ onto $\Lambda$.
The following straightforward lemma will be helpful:
\begin{lema}\label{lema:null_lags}
	If $F\colon \mb V \to \R$ is $\mc A$-quasiaffine then the composition $F\circ T$ is $\D^k$-quasiaffine; the converse also holds if $\tp{span}\,\Lambda=\mb V$.
\end{lema}

\begin{proof}
	We only prove the converse direction as the other one is absolutely similar, so
	suppose that $F$ is $\D^k$-quasiaffine.
	By assumption, for each $v\in \mb V$ there is some $z\in \odot^k (\R^n,\mb U)$ such that $T z =v$.
	Then for any $u \in C^\infty_c(\Omega, \mb U)$ we have
	$$F(v)=F\circ T( z)=\fint_\Omega F\circ T(z + \D^k u)\d x=\fint_\Omega F(v+\mc B u)
	\d x$$
	where we used the linearity of $T$ and (\ref{eq:C}).
	This shows that $F\circ T$ is $\mc A$-quasiaffine.
\end{proof}

\begin{remark}
	An interesting takeaway from this lemma is that there seem to be two competing notions of polyconvexity \cite{Boussaid2018}. We follow the usual definition in the curl-free case \cite{Ball1977} and say that $F\colon \mb V\to \R$ is $\mc A$-\textbf{polyconvex} if it is the pointwise supremum of $\mc A$-quasiaffine functions; this is an intrinsic notion. Another possibility is to consider the class of functions $F$ such that $F\circ T$ is $\D^k$-polyconvex. This class is contained in the class of $\mc A$-quasiconvex functions, as one readily checks by a calculation similar to the one in the proof of the lemma. Let us call such functions \textbf{extrinsically $\mc A$-polyconvex}. We have that
	$$\tp{convexity}\, \implies \, \mc A\tp{-polyconvexity }\, \implies \, \tp{ extrinsic } \mc A\tp{-polyconvexity} \, \implies \, \mc A\tp{-quasiconvexity}$$
	and in some cases the first two notions coincide, see Example
	\ref{ex:sym} below, where $\mc B=\mc E$. In this case, $F(\mc B u)=\det \mc E u$ is extrinsic symmetric polyconvex, but not symmetric polyconvex. It is also clear that the intrinsic and extrinsic
	classes of polyconvex integrands can be the same, as it is the case
	when $\mc B=\D^k$. These notions have been further studied in the particular case
	where the integrands depend on differential forms \cite{Bandyopadhyay2016}.
\end{remark}

Since we assume that $\tp{span}\, \Lambda=\mb V$, we have that $T$ is onto $\mb V$ and the Rank--Nullity Theorem yields the linear isomorphism
\begin{equation}
\label{eq:id}
\odot^k (\R^n,\mb U)\cong \ker T\oplus \tp{im}\, T=\ker T\oplus\mb V.\end{equation}
Therefore we think of $\mb V$ as a subspace of 
$\odot^k (\R^n,\mb U)$ and of $T$ as a projection onto that subspace. The utility of this viewpoint is illustrated by the previous results: under the assumptions of the lemma, the map $F\circ T$ is an affine combination of Jacobians and under the identification (\ref{eq:id}) we can in fact think of $F$ as real-valued map defined on 
$\mb V\subseteq \odot^k (\R^n,\mb U)$. Thus, we have shown:

\begin{prop}\label{prop:explicitnulllags}
	Let $F\colon \mb V \to \R$ be a $\mc A$-quasiaffine map. Then, under the identification (\ref{eq:id}), we can find constants $
	c_M\in \R$ such that
	\begin{equation}
	F\circ T=F(0)+
	\sum_{M}c_M M,
	\label{eq:jacobians}
	\end{equation}
	where $M\colon \R^{N\times n}\rightarrow \R$, $N=\dim\odot^{k-1} (\R^n,\R^m) $, runs over all minors of $N\times n$ matrices.
\end{prop}

In other words, in the right coordinates, $\mc A$-quasiaffine maps are precisely the Jacobians.

It is natural to ponder for a moment whether one can hope for a more
invariant statement. The crucial point here is that proper minors,
i.e.\ minors which are not the determinant, have no intrinsic
geometric content, in the sense that they are not invariant under
changes of coordinates. We make this well-known fact very precise in the following remark.

\begin{remark}
	Assume that $m\neq n$. A (non-trivial) linear isomorphism $L\in \tp{GL}( \R^{m}\otimes \R^n )$ maps minors into minors, i.e.\ 
	$M \circ L\colon \R^{m\times n} \cong \R^m\otimes \R^n\to \R$ is a minor whenever $M\colon \R^{m\times n}\to \R$ is a minor, if and only if 
	\begin{equation}L=R\otimes S\hs
	\tp{ for some }
	R\in \tp{GL}(m), S\in \tp{GL}(n).
	\label{eq:T}
	\end{equation} This follows from the fact that minors are precisely the rank-one affine functions (and that they are affine only along rank-one lines) and that $T$ maps the rank-one cone into itself if and only if it has the form (\ref{eq:T}), see \cite[Theorem 1]{Marcus1959}. This shows the intuitive fact that minors are closely tied with the tensor product structure of the vector space $\R^{m}\otimes \R^{n}$ and that, to make sense of them, one should not forget this structure and think of it instead as a generic vector space of dimension $m\times n$.
\end{remark}

\begin{remark}
	\textsc{Robbin}--\textsc{Rogers}--\textsc{Temple} \cite[\S 5.2]{Robbin1987} asked whether all weakly continuous functions could be obtained in a framework with differential forms. Proposition
	\ref{prop:explicitnulllags} gives a positive answer to this question under the main assumption (\ref{eq:assumption}). We refer the reader to the works
	\cite{Iwaniec1993,Sil2019} for further
	properties of null Lagrangians depending on differential forms. 
\end{remark}

The above discussion shows that the choice of coordinates (\ref{eq:id}) is in some sense very arbitrary. Nonetheless, the identification (\ref{eq:id}) also turns out to be computationally effective. The computational problem is to decide which, if any, of the constants 
 $c_M$ that appear in (\ref{eq:jacobians}) can be taken to be non-zero. 
The key to solving this problem is the the immediate fact that, if $H\colon \odot^k (\R^n,\mb U)\to \R$ denotes the right-hand side of (\ref{eq:jacobians}), then
$$H=H\circ T.$$
We think of both sides of this equality as being polynomials in the algebraically independent variables $x_{i_1,\dots, i_k}$, $i_j\in \{1, \dots, n\}$,  that define an element  $X=(x_{i_1,\dots, i_k})\in \odot^k (\R^n,\mb U)$. Since both sides are equal as polynomials, all the coefficients must be the same. Noting that the coefficients of these polynomials depend linearly on 
 $(c_M)_M$, we find from the equality of coefficients a linear system 
for the 
 $c_M$ whose solution determines completely the possible null Lagrangians.
This system can in turn be solved using symbolic 
computation software. 
One can also  fix a specific order of the minors in (\ref{eq:jacobians}), say $s$, and solve instead the above system with $H$ replaced by 
$$H_s\equiv 
\sum_{\deg M=s} c_M M
$$
since minors of different orders cannot cancel each other out.
For the sake of concreteness, we illustrate this method with  simple examples. 

\begin{ex}\label{ex:sym}
	Let $T=\tp P_\tp{sym}$, where $\tp P_\tp{sym}\colon
        \R^{n\times n} \to \R^{n\times n}_\tp{sym}=\mb V$ is the
        orthogonal projection, i.e.\ $\mc B=\mc E$ is the symmetric
        gradient. The algorithm described above can be very easily
        implemented; in \texttt{Mathematica} a possible
        implementation is given in Code Listing \ref{code}.
	\begin{figure}[thp]
          \renewcommand{\figurename}{Code Listing}
		\begin{CenteredBox}
		\begin{lstlisting}
	sym[X_] := (X+Transpose[X])/2;
	X = Array[Subscript[x, #1, #2]&, {n, n}];
	const = Array[Subscript[c, #]&, Binomial[n, s]^2];
			
	Solve[
	   DeleteCases[
	      CoefficientList[
	         const.(Flatten[Minors[sym[X], s] - Minors[X, s]]), Flatten[X]
	      ]//Flatten, 0
	   ] == 0
	]
	\end{lstlisting}
		\end{CenteredBox}
                                \caption{A possible implementation of
                                  the algorithm in the setup of
                                  Example \ref{ex:sym}}
                \label{code}
	\end{figure}
	
	In this case, however, it is relatively easy to verify analytically that there are no  non-affine null Lagrangians (when $n=2, 3$, this was proved in \cite{Boussaid2018} as a consequence of more general statements). For this, it suffices to consider the case where the null Lagrangians are homogeneous polynomials of degree 2. Indeed, if $F$ is an $s$-homogeneous null Lagrangian  then $\p F/\p v$ is an $(s-1)$-homogeneous null Lagrangian, where $v$ is any vector from $\mb V$; this follows straightforwardly from (\ref{eq:B-qa}). Thus, if we prove that there are no null Lagrangians with order two then there can be no higher order null Lagrangians.
	
	From the relation $H_2=H_2 \circ T$ we deduce that, for any $X\in \R^{n\times n}$,
	$H_2(X)=H_2(X^T).$
	Given a $2\times 2$ minor $M$, let $\widetilde M$ be the minor defined by $\widetilde M(X)\equiv M(X^T)$; in particular $\widetilde M=M$ if $M$ is a principal minor. For the sake of concreteness, let us say that $M(X)=\det[(x_{i,j})_{i\in I, j \in J}]$ where $I=\{i_1,i_2\},J=\{j_1,j_2\} \subset \{1,\dots, n\}$. If we let  $X=(x_{i,j})$ be such that
	$$x_{i,j}=\begin{cases}
	1 & (i,j)=(i_k,j_k) \tp{ for } k\in \{1,2\}\\
	0 & \tp{otherwise}
	\end{cases}$$
	then 
	$$c_M=c_M M(X)=H_s(X)=H_s(X^T)=c_{\widetilde M} \widetilde M(X)=c_{\widetilde M}.$$ Now let $Y=X-X^T$
	and observe that, since $M(Z)=M(-Z)$ for any $Z\in \R^{n\times n}$,
	$$c_M +c_{\widetilde M} =c_M M(X)+c_{\widetilde M} \widetilde M(X)=H_2(Y)= H_2 (T( Y))=0.$$
	The conclusion follows.
\end{ex}

\begin{ex}
	Another relevant example is that of solenoidal matrix fields, i.e.\ $\mc A=\tp{div\,}$, which can be embedded in the framework of exterior derivatives of differential forms. As above, we are particularly interested in null Lagrangians of degree (at least) two. We will consider divergence-free fields $v\colon\R^n\rightarrow \R^{n\times n}$ for $n=2,3$. For $n=2$, we can set\footnote{In this simple case, an example of a potential operator $\mc B$ is easily chosen by inspection.} 
	$$
	v=\left(
	\begin{matrix}
	\partial_2 u_1 &-\partial_1 u_1\\
	\partial_2 u_2 &-\partial_1 u_2
	\end{matrix}
	\right)=(\D u)J,\quad\text{where }
	J=\left(
	\begin{matrix}
	0&-1\\
	1&0
	\end{matrix}
	\right),
	$$
	and note that $H_2=c\det$ for $c\in \R$. To see that this is indeed a null Lagrangian, we need only observe that $\det X=\det (XJ)$ for $X\in\R^{2\times 2}$.
	
	For $n=3$, we will show that there are no (homogeneous) quadratic null Lagrangians. First, recall that $\tp{curl}$ is a potential operator for $\tp{div}$ in this case, which we write in the  form
	$$
	\mc Bu\equiv \tp P_{\tp{asym}}\D u,\quad\text{for }u\colon \R^3\rightarrow\R^3,
	$$
	where $T\equiv \tp P_{\tp{asym}}$ denotes the orthogonal
        projection onto anti-symmetric matrices. We will test the
        relation $H_2(X)=H_2(T(X))$ with different matrices
        $X\in\R^{3\times 3}$ to show that $H_2=0$, since this is
        enough to show that there are no non-affine null Lagrangians
        (see also Example~\ref{ex:sym}). First, note that by taking
        $X=e_i\otimes e_i+e_j\otimes e_j$, $i\neq j$, the coefficients
        of the principal minors in $H_2$ must be zero. The other
        $2\times 2$ minors touch the main diagonal on exactly one
        entry, say $(i,i)$. By taking $X=ae_i\otimes e_i+e_j\otimes
        e_k$, $j\neq i\neq k$, for $a\in \R$, we see that indeed $H_2=0$.
	
	For general dimension $n\geq 3$, it is not too difficult to see that there are no non-affine div-null Lagrangians.
\end{ex}

It would  be interesting to give a theoretical characterization of the solutions of the computational problem. This is also a relevant question since the linear system  described above grows factorially in $\tp{dim}\,\mb V$, although in applications to continuum mechanics this number  is usually relatively small.
Unfortunately, even in the special case when $\mc B$ has order one such a characterization seems difficult.  The authors were unable to give a definitive answer even to the following simple-looking question.

Assume we are given a projection $T\colon \R^{m\times n}\to \mb V$, which can be chosen to be orthogonal, onto
some subspace $\mb V\subseteq \R^{m\times n}$. Consider a function $H_s\colon \R^{m\times n}\to \R$ as above, i.e.\ 
$$H_s(X)=\sum_{\mathrm{deg}M=s} c_M M(X),\hs  H_s=H_s\circ T$$
where the sum runs over the set of all minors (not necessarily principal) of order $2\leq s< \min\{m,n\}$.  The second condition can be equivalently rewritten as
\begin{align}\label{eq:question}
H_s(X)=H_s(X+Y) \textup{ for all } X, \tp{ all }Y \tp{ such that } T(Y)=0.
\end{align}
We think of this as saying that the linear combination of minors $H_s$
only depends on the coordinates of $\mb V$.

\begin{question}
	Under which conditions on $T$ can we find non-zero $H_s$ satisfying
	\eqref{eq:question}? Can we characterize such $H_s$ in terms of $\mb V$?
\end{question}

\section{Compensated compactness in Hardy spaces}
\label{sec:hardy}

We begin by stating the main theorem of this section; as usual, we assume (\ref{eq:assumption}) holds throughout. Recall that $\mc A$-quasiaffine maps are polynomials (c.f.\ Proposition~\ref{prop:abstractnl}) and see Definition~\ref{def:cocan} for the definition of $\mb I_\mc A.$

\begin{teo}\label{teo:hardy}
	Let $F\colon \mb V\to \R$ be locally bounded and Borel. If the implication
	\begin{equation}
	v\in C^\infty_{c,\mc A}(\R^n)
	\hs \implies \hs F(v)\in \mathscr
	H^1(\R^n)\label{eq:hardyimp}
	\end{equation}
	holds, $F$ is a sum of homogeneous $\mc A$-quasiaffine functions of degree at most $\min\{n,\dim \mb V\}$.

	Conversely, assume that $F$ is an $s$-homogeneous $\mc A$-quasiaffine function. If $s\geq 2$ then (\ref{eq:hardyimp}) holds and moreover 
	$$\Vert F(v)\Vert_{\mathscr H^1}\leq C \Vert v\Vert^s_{L^s}\quad\text{for }v\in C^\infty_{c,\mc A}(\R^n).$$
	If $s=1$, we have that $F(v)=v_0\cdot v$ for some $v_0\in\mb
        V$ and \eqref{eq:hardyimp} holds if and only if $v_0\perp\mb
        I_{\mc A}$, although nonetheless the above estimate fails.
\end{teo}

It will be convenient to prove the homogeneous case first. We will deal with the linear case, which is somewhat degenerate, afterwards.

\begin{prop}\label{prop:auxh1}
	Let $F$ be a homogeneous polynomial on $\mb V$ of degree $2\leq s\leq\min\{ n,\dim\mb V\}$. The
	following statements are equivalent:
	\begin{enumerate}
		\item\label{itm:canc} $\int_{\R^n}F(v)=0$ for all
		$v\in C_{c,\mc A}^\infty(\R^n)$.
		\item\label{itm:hardy} $\|F(v)\|_{\mathscr H^1(\R^n)}\leq
		c\|v\|^s_{L^s(\R^n)}$ whenever $v\in L^s_{\mc A}(\R^n)$.
	\end{enumerate}
\end{prop}

Observe that the direction \ref{itm:hardy} $\Rightarrow$ \ref{itm:canc} is clear, since functions in $\mathscr H^1(\R^n)$ have zero mean.
To prove the estimate, we follow the original strategy in \cite{Coifman1993}. In fact, we will use the potential $\mc B$ and Lemma~\ref{lema:null_lags} to show that the estimate can be inferred from the case $\mc B=\D^k$. The statement for $v=\D^k u$ is then known from \cite[Theorem~6.2]{Lindberg2017}; here we give a proof by reduction to the div-curl case.

We emphasize the technical fact that the assumption $s\leq n$ will be important in order to apply the Poincar\'e--Sobolev inequality. Given a ball $B_t(x)\subset \R^n$ we write
$(f)_{x,t}\equiv \fint_{B_t(x)} f(y) \d y.$

\begin{proof}[Proof of Proposition \ref{prop:auxh1}]
	From \eqref{eq:spaces} we see that it is sufficient to bound $F(\mc Bu)$ for $u\in C^\infty_c(\R^n,\mb U)$. Recalling Lemma~\ref{lema:null_lags}, it is natural to first deal with the case $\mc B =\D^k$. This case is already known from \cite{Lindberg2017}, but here we give a simpler proof, at least as far as notation is concerned. 
	
	We claim that if $\int_{\R^n} F(\D^ku)\d x=0$ for $u\in C^\infty_c(\R^n)$ then there is an estimate 
	\begin{align}\label{eq:pula}
	\|F(\D^k u)\|_{\mathscr H^1}\leq C\|\D ^ku\|_{L^s}^s\text{ for }u\in C^\infty_c(\R^n). 
	\end{align}
	The assumption implies that $F$ is $\D^k$-quasiaffine at zero, and hence everywhere, c.f.~the proof of Theorem~\ref{teo:hardy} below. By Theorem~\ref{teo:BCO} and $s$-homogeneity, we see that $F$ is a linear combination of minors of order $s$ of $\D U$, where $U\equiv D^{k-1}u$, i.e.\
	$$
	F(\D^{k}u)=\sum_{\tp{deg\,}M=s}c_M M(\D U).
	$$
	Thus, it is sufficient to prove the estimate in the case $F(\D^{k}u)=M(\D U)$. We choose coordinates $x=(x^\prime,x^{\prime\prime})\in\R^n$ and  $T=(T^\prime,T^{\prime\prime})\in \odot^{k-1}(\R^n,\mb U)$, where $x^\prime,\;T^\prime$ are $s$-dimensional. Then we can write
	$$
	M(\D U(x))=\det \D_{x^\prime}U^\prime (x).
	$$
	Note that $\D_{x^\prime}$ can be regarded as a differential operator on $\R^n$.
	
	To prove the claim, one can use the reasoning used in the proof of \cite[Theorem~II.1.1)]{Coifman1993}. By looking at the $(1,1)$ entry of the identity $(\det A)\tp{Id}=A(\tp{cof\,}A)^T$ applied to $A=\D f$, $f\colon\R^s\rightarrow\R^s$, we see that $\det \D f=\D f_1\cdot \sigma$, where $\sigma$ is the first row of the matrix $\tp{cof}\,\D f$, which is row-wise divergence-free, and moreover we have the pointwise estimate $|\sigma|\lesssim |\D f_2||\D f_3|\ldots |\D f_s|$. In our case, it is elementary to adapt these considerations to see that
	$$
	M(\D U)=\langle\D_{x^\prime}U^\prime_1(x),\Sigma(x)\rangle_{\R^s}
	$$
	where $\langle \cdot, \cdot \rangle_{\R^s}$ is the usual Euclidean inner product and $\Sigma\colon\R^n\rightarrow\R^s$ is such that 
	$$\tp{div}_{x^\prime}\Sigma=0\tp{ in }\R^n \hs \tp{ and } \hs |\Sigma|\lesssim|\D U^\prime_2||\D U^\prime_3|\ldots |\D U^\prime_s|.$$
	Here $\tp{div}_{x^\prime}=\D_{x^\prime}^*$ is the adjoint of the differential operator $\D_{x^\prime}$.
	
	Let $\psi\in C^\infty_c(B_1(0))$ be a non-negative function with non-zero mean. We have
	\begin{align*}
	|\psi_t*M(\D U)|(x)&=\left|\frac{1}{t^{n}}\int_{\R^n}\psi\left(\frac{x-y}{t}\right)\langle\D_{x^\prime}U^\prime_1(y),\Sigma(y)\rangle_{\R^s}\dif y\right|\\
	&=\left|\frac{1}{t^{n}}\int_{B_t(x)}\biggr\langle\D_{x^\prime}\big[U^\prime_1(y)-(U^\prime_1)_{x,t}\big],\psi\left(\frac{x-y}{t}\right)\Sigma(y)\biggr\rangle_{\R^s}\dif y\right|\\
	&=\left|\frac{1}{t^{n+1}}\int_{B_t(x)}(U^\prime_1(y)-(U^\prime_1)_{x,t})\biggr\langle (\D_{x^\prime}\psi)\left(\frac{x-y}{t}\right),\Sigma(y)\biggr\rangle_{\R^s}\dif y\right|\\
	&\lesssim
	\frac{1}{t^{n+1}}\int_{B_t(x)}|U^\prime_1(y)-(U^\prime_1)_{x,t}||\Sigma(y)|\dif y,
	\end{align*}
	where in the third equality we integrated by parts, using the fact that that $\tp{div}_{x'}\Sigma=0$. We apply H\"older's inequality with $p=nq/(n+q)$ for some $q\in(1,s)$ to get
	\begin{align*}
	|\psi_t*M(\D U)|(x)&\lesssim \frac{1}{t}\left(\fint_{B_t(x)} |U^\prime_1(y)-(U^\prime_1)_{x,t}|^p\dif y\right)^{1/p}\left(\fint_{B_t(x)} |\Sigma(y)|^{p^\prime}\dif y\right)^{1/{p^\prime}}\\
	&\lesssim
	\left(\fint_{B_t(x)} |\D U^\prime_1(y)|^{q}\dif y\right)^{1/q}\left(\fint_{B_t(x)} |\Sigma(y)|^{p^\prime}\dif y\right)^{1/{p^\prime}}
	\end{align*}
	where we also used the Poincar\'e--Sobolev inequality; note that the implicit constant does not depend on $t$.
	We further ensure that $p^\prime=p/(p-1)<s/(s-1)=s'$ by requiring $q>ns/(n+s)$. We next note that, writing $\mc M$ for the Hardy--Littlewood maximal function,
	\begin{align*}
	\sup_{t>0}|\psi_t*M(\D U)|(x)&\lesssim\sup_{t>0}\left[\left(\fint_{B_t(x)} |\D U^\prime_1(y)|^{q}\dif y\right)^{1/q}\left(\fint_{B_t(x)} |\Sigma(y)|^{p^\prime}\dif y\right)^{1/{p^\prime}}\right]\\
	&\leq\sup_{t>0}\left(\fint_{B_t(x)} |\D U^\prime_1(y)|^{q}\dif y\right)^{1/q}\sup_{t>0}\left(\fint_{B_t(x)} |\Sigma(y)|^{p^\prime}\dif y\right)^{1/{p^\prime}}\\
	&=\mc M(|\D U^\prime_1|^q)(x)^{1/q}\mc M(|\Sigma|^{p^\prime})(x)^{1/p^\prime}
	\end{align*}
	Integrating this estimate with respect to $x$ and applying H\"older's inequality twice we obtain
	\begin{align*}
	\|M(\D U)\|_{\mathscr H^1(\R^n)}&
	\lesssim\left(\int_{\R^n}\mc M(|\D U^\prime_1|^q)^{s/q}\right)^{1/s}\left(\int_{\R^n}\mc M(|\Sigma|^{p^\prime})^{s^\prime/p^\prime}\right)^{1/s^\prime}\\
	&\lesssim\left(\int_{\R^n}|\D U^\prime_1|^s\right)^{1/s}\left(\int_{\R^n}|\Sigma|^{s^\prime}\right)^{1/s^\prime}\\
	&\lesssim\left(\int_{\R^n}|\D U^\prime_1|^s\right)^{1/s}\left(\int_{\R^n}\prod_{i=2}^s|\D U^\prime_i|^{s/(s-1)}\right)^{(s-1)/s}\\
	&\leq\prod_{i=1}^s\|\D U_i^\prime\|_{L^s(\R^n)}\leq C\|\D U^\prime\|^s_{L^s(\R^n)},
	\end{align*}
	where moreover the second inequality follows by boundedness of the maximal function.
	This proves the desired claim \eqref{eq:pula}.
	
	To conclude the proof, we return to the case of a general $\mc B$ and use Lemma~\ref{lema:null_lags}:
	$$
	\|F(\mc B u)\|_{\mathscr H^1(\R^n)}=\|F\circ T(\D^k u)\|_{\mathscr H^1(\R^n)}\leq C\|\D^k u\|_{L^s(\R^n)}^s\leq C\|\mc B u\|_{L^s(\R^n)}^s,
	$$
	where the last estimate follows from Theorem~\ref{teo:Lpest}, since the left-hand side is kept unchanged by replacing $u$ with $u-P_{\mc B} u$.  
\end{proof}

\begin{remark} It is possible to give a more abstract proof of the above proposition in the spirit of \cite{Lindberg2017,Strzelecki2001}, circumventing the explicit representation of null Lagrangrians from \cite{Ball1981}. The basic idea is that, since both $F$ and $\mc B$ are homogeneous, we can write
	$$F(\mc B u)=\sum_{\nu \in \{1,\dots,\tp{dim}\mb U\}^s} \sum_{|\beta_1|, \dots, |\beta_s|=k} f_{\beta,\nu}\prod_{i=1}^s
	\p^{\beta_i} u^{\nu_i}$$
	for some constants $f_{\beta,\nu}\in \R$,  where each $\beta_i$ is an $n$-multi-index. Using the Leibniz rule together with the cancellation assumption \ref{itm:canc} we have, after some elementary calculations,
	\begin{equation*}
	\int_{\R^n}
	\psi_t(x -y) F(\mc Bu(y))\d y
	=-
	\sum_{\beta,\nu} \frac{f_{\beta,\nu}}{t^n} \int_{\R^n}\prod_{i=1}^s 
	\sum_{\gamma<\beta} c_{\beta,\gamma}\p^{\beta_i-\gamma_i}\phi\left(\frac{x-y}{t}\right) \,\p^{\gamma_i}
	u^{\nu_i}(y)\d y                           
	\end{equation*}
	where by $\gamma<\beta$ we mean that there is some $i$ such that $\gamma_i<\beta_i$ as multi-indices and $\psi\equiv \phi^s$. The point is that, for each $(\beta,\nu)$ fixed, at least one of the terms on the right has one less derivative than the others. Therefore, subtracting enough moments from $u$, we see from the Poincar\'e--Sobolev inequality that this term has higher integrability than the others. One then concludes by suitably applying H\"older's inequality, similarly to above.
\end{remark}

In order to deduce the theorem from the proposition we need to justify the assumption $s\geq 2$.
This will be done in the following lemma, which proves a non-inclusion of $L^1_{\mc A}(\R^n)$ into $\mathscr{H}^1(\R^n)$ and which is somewhat reminiscent of the much deeper Ornstein's non-inequality \cite{Ornstein1962,Kirchheim2016}. The common theme is, of course, the lack of boundedness of singular integrals on generic subspaces of $L^1$, c.f.\ Proposition \ref{prop:Riesztransfcharact}. Recall that we assume (\ref{eq:assumption}).

\begin{lema}
	\label{lema:Afreefieldshardy}
	Let $v_0\in\mb V$ be a non-zero vector. Then there exists a sequence $v_j\in C_{c,\mc A}^\infty(\R^n)$ such that $\|v_0\cdot v_j\|_{\mathscr H^1}\geq j$ but $\|v_j\|_{L^1}\leq 1$ for all $j\geq 1$.
\end{lema}
\begin{proof}
	
	By the spanning cone condition, there exists non-zero $\tilde v_0\in \mb V$ and $\xi \in \R^n$ such that $\tilde v_0\in\ker\mc A(\xi)=\tp{im\,}\mc B(\xi)$, say $\mc{B}(\xi)u_0=\tilde v_0$, and $\tilde v_0\cdot v_0\neq 0$. Note that if $u(x)=f(x\cdot\xi)u_0$ for some $f\in 
	L^1_{\tp{loc}}(\R)$, then $\mc B u(x)=f^{(k)}(x\cdot\xi) \mc B(\xi)u_0$. In particular, by choosing $f(t)=\max\{t,0\}^{k-1}$, we obtain that $\mc B u=(k-1)!\tilde{v}_0\left(\mathscr H^{n-1}\mres \{x\cdot\xi=0\}\right)$.
	
	By defining $\tilde u=\rho u$ for some test function $\rho$ that equals one in a neighbourhood of the unit ball, we obtain a compactly supported $\mc A$-free measure $\mc B\tilde u$ that is not absolutely continuous.
	
	We now explain how the proof can be concluded easily. Assume for contradiction that the claim of the lemma fails, so that there is a bound
	$$
	\|v\cdot v_0\|_{\mathscr H^1}\leq C\|v\|_{L^1}\quad\text{for }v\in C^\infty_{c,\mc A}(\R^n).
	$$
	Consider a sequence of mollifications $\tilde u_\varepsilon$, so that $\tilde u_\e\in C^\infty_{c}(\R^n)$ and $\mc B \tilde u_\varepsilon\wstar \mc B\tilde u$ as measures. The estimate implies
	$$
	\|\mc B\tilde u_\varepsilon\cdot v_0\|_{\mathscr H^1}\leq C\sup_{\varepsilon\in(0,1)}\|\mc B \tilde u_\varepsilon\|_{L^1}<\infty,
	$$
	and so, up to subsequences, $(\mc B u_\varepsilon\cdot v_0)_\varepsilon$ is convergent in $\mathscr H^1$. 
	It follows that $\mc B u\cdot v_0\in\mathscr H^1$, so $\mc B u\cdot v_0$ is absolutely continuous, which leads to a contradiction since $\tilde v_0\cdot v_0\neq 0$.  
\end{proof}

We are finally ready to finish the proof.

\begin{proof}[Proof of Theorem \ref{teo:hardy}]
	Note that if (\ref{eq:hardyimp}) holds then $F$ is $\mc A$-quasiaffine at zero, i.e.\ we have (\ref{eq:B-qa}) with $z=0$:
	if $u\in C^\infty_c(\R^n, \mb U)$ then $\mc B u \in C^\infty_{c,\mc A}$ and therefore
	$\int_{\R^n} F(\mc B u)=0$
	since functions in the Hardy space have zero mean. Moreover, if $F$ is $\mc A$-quasiaffine at zero then it is quasiaffine everywhere. 
	To see this, fix $z \in \mb V$ and $u\in C^\infty_c(\R^n, \mb
        U)$. Let $\phi\in C^\infty_c(\R^n, \mb U)$ be chosen so that
        $\mc B \phi=z$ in the support of $u$; thus $\int_{\R^n} F(t
        \mc B\phi + \mc B u)=0$ for any $t\in \R$. Then, since $F(t \mc B \phi+\mc B u)=F(t \mc B \phi)$ outside the support of $u$,
	\begin{align*}
	0&=\frac{\d}{\d t} \int_{\R^n} F(t \mc B \phi + \mc B u)- F(t \mc B \phi) \d x  = 
	\frac{\d}{\d t}\int_{\R^n} F(t z + \mc B u)-F(t z) \d x
	\end{align*}
	so the right-hand side is constant. In particular, comparing the values at $t=1$ and $t=0$,
	$$\int_{\R^n} F(z+ \mc B u)-F(z) \d x = \int_{\R^n} F(\mc B u)\dif x=0,$$
	as wished.
	
	Since $F$ is $\mc A$-quasiaffine, it is a polynomial, which we write as a sum of homogeneous terms as $F=\sum_{s=0}^n P_s$. In fact, it is clear that $P_0=0$.  We note that
	$$
	0=\int_{\R^n}F(t\mc Bu)\dif x=\sum_{s=1}^n t^s\int_{\R^n}P_s(\mc Bu)\dif x
	$$
	for all $t\in \R$ and  $u$ fixed. This implies that each $P_s$ is $\mc A$-quasiaffine as well.
	
	Conversely, if $F$ is $\mc A$-quasiaffine then it is continuous and, given $v\in C^\infty_{c,\mc A}(\R^n)$ we have, from Proposition \ref{prop:hodge}, a sequence $u_j \in C^\infty_c(\R^n,\mb U)$ such that $\mc B u_j \to v$ in $L^p(\R^n, \mb V).$
	Therefore
	$$0=\int_{\R^n} F(\mc B u_j)\d x \to \int_{\R^n} F(v) \d x \hs \tp{ as } j\to \infty,$$
	so we can use Proposition~\ref{prop:auxh1} to see that (\ref{eq:hardyimp}) and the required estimate for $s$-homogeneous $F$, $s\geq2$, holds. 
	
	Finally, let $F$ be linear, say $F(v)=v_0\cdot v$. By Lemma~\ref{lema:Afreefieldshardy}, there can be no uniform estimate in this case. Moreover, if $v_0$ is not orthogonal to $\mb I_{\mc A}$, we consider $v_1\in \mb I_{\mc A}$ be such that $v_0\cdot v_1\neq 0$ and a scalar test field $\rho\in C^\infty_{c}(\R^n)$ with non-zero integral. Then $\rho v_1\in C_{c,\mc A}^\infty(\R^n)$ but {$F(\rho v_1)$} is not in the Hardy space.
	On the other hand, if $v_0$ is orthogonal to $\mb I_{\mc A}$, we write $v=v_1+v_2$ for the decomposition of $v\in C_{c,\mc A}^{\infty}(\R^n)$ such that $v_1\in C^\infty_c(\R^n,\mb I_{\mc A})$ and $v_2\in C_{c,\tilde{\mc A}}^{\infty}(\R^n)$ (recall Lemma~\ref{lema:cocan} and its notation). We then have that $F(v)=v_0\cdot v_2$, which is a test function with zero integral, as is $v_2$ by Lemma~\ref{lema:char_cocanc}. It follows that $F(v)$ lies in $\mathscr H^1(\R^n)$. The proof is complete. 
\end{proof}

We remark that Theorem \ref{teo:hardy} seems to contradict \cite[Proposition 6.3]{Lindberg2017}, but unfortunately there appears to be a mistake in the calculation presented there.
As a simple consequence of the theorem, we have:
\begin{cor}
	If $F$ is an $s$-homogeneous $\mc A$-null Lagrangian, $s\geq 2$, then 
	$$F\colon (L^s_\mc A(\R^n),\tp{w})\to (\mathscr H^1(\R^n), \tp{w}^*) \tp{ is sequentially continuous.}$$
\end{cor}
\begin{proof}
	Given a sequence $v_j \in L^s_\mc A(\R^n)$ such that $v_j \w v$ in $L^s$, we have from \ref{it:improved} of Proposition \ref{prop:abstractnl} that
	$$\int_{\R^n}\varphi F(v_j)  \d x \to \int_{\R^n} \varphi F(v)  \d x \hs \tp{for all } \varphi \in C^\infty_c(\R^n).$$
	Since $F(v_j), F(v)$ are uniformly bounded in $\mathscr H^1(\R^n)$, and by density of test functions in $\tp{VMO}(\R^n)$, we can replace  $C^\infty_c$ by $\tp{VMO}$ above; in this case, the integrals should be thought of as shorthand notation for the duality pairing. 
\end{proof}

The utility of Hardy space bounds when dealing with weakly converging sequences is  apparent, for instance, from Theorem \ref{teo:jonesjourne}.
To conclude this section we provide some concrete examples which illustrate the way in which Theorem \ref{teo:hardy} contains the examples of \cite{Coifman1993}. 

\begin{ex}[Stationary Maxwell system]
	Let $E,B \in C^\infty_c(\R^n,\R^n)$ be such that
	\begin{equation*}
	\tp{div}\, E =0, \hs \tp{curl}\, B = 0.\label{eq:maxwell}
	\end{equation*}
	Then the vector field $(E,B)$ is $\mc A$-free, where of course $\mc
	A=(\tp{div},\tp{curl})$, which is a constant rank operator. The
	quantity $E\cdot B$ is easily seen to be $\mc A$-quasiaffine: indeed,
	writing $B=\D u$ for some smooth $u$, 
	$$\int_{\R^n}E(x)\cdot B(x) \d x=  -\int_{\R^n} u(x)\,\tp{div\,} E(x)  \d x=0.$$
	Therefore, from the theorem,
	$$\Vert E\cdot B\Vert_{\mathscr H^1} \lesssim \Vert (E,B)\Vert_2.$$
	In particular, and arguing by density, we see that the same holds if $B,E\in L^2(\R^n,\R^n)$.
\end{ex}

A generalization of the previous example for quadratic forms was given in \cite{Li1997}, even without assuming that $\mc A$ has constant rank.

\begin{ex}[Double cancellation]
	Let us take vector fields $U,V\in L^2(\R^n,\R^{n\times n})$; again we
	shall first argue formally as the general case can be recovered by density.
	We introduce the constant rank operator  
	$$\mc A
	\begin{bmatrix}
	U \\ V
	\end{bmatrix}
	=
	\begin{bmatrix}
	\D(\tp{tr}\, U)\\ \tp{curl}\, U\\
	\tp{curl}\, V
	\end{bmatrix}.
	$$
	Note that an $\mc A$-free test vector field $(U,V)$ can be written as $U=\D u$ and $V=\D v$, where moreover $\tp{div}\, u=0$ since $\tp{div}\, u=\tp{tr}\, U$ is both constant and zero outside a compact set.
	The function $F(U,V)=\langle U^T,V\rangle =\sum_{i,j} U^{j,i} V^{i,j}$
	is $\mc A$-quasiaffine:
	$$\int_{\R^n} F(U,V)= \int_{\R^n} \sum_{i,j}\p^i u^j \p^j v^i=\int_{\R^n} \tp{div}\, u
	\tp{ div}\, v=0.$$
	Therefore, from the theorem,
	$$\bigg\Vert \int_{\R^n} \sum_{i,j}\p^j u_i \p^i v_j \bigg\Vert_{\mathscr H^1}\lesssim
	\Vert Du\Vert_2 \Vert Dv\Vert_2$$
	whenever $u$ is divergence-free.
\end{ex}

\begin{ex}[Monge-Amp\`ere]
	Let $\mc A$ be an annihilator for $\D^2$. Given
	$U,V\in C^\infty_c(\R^2,\R^2)$, the map
	$$F(U,V)=U_{11} V_{22} + U_{22} V_{11} - 2 U_{12} V_{12}$$
	is $\mc A$-quasiaffine: writing $U=D^2 u, V=D^2 v$, we have
	$$\int_{\R^n} F(U,V) = \int_{\R^n} \p_{xx} u \p_{yy} v + \p_{yy} u
	\p_{xx} v - 2 \p_{xy} u \p_{xy} v\equiv \int_{\R^n} [u,v]=0$$
	by integration by parts. Thus
	$$\Vert [u,v]\Vert_{\mathscr H^1}\lesssim \Vert D^2 u \Vert_2 \Vert D^2 v \Vert_2.$$
\end{ex}

\section{Continuity estimates for null Lagrangians}
\label{sec:estimates}

In the case where $\Omega=\R^n$ it is possible to give a simple proof of weak continuity of null Lagrangians following the strategy from \cite{Brezis2011b,Iwaniec2001}. This proof circumvents the use of Theorem \ref{teo:lsc} and moreover has the advantage of giving a quantitative statement.

\begin{prop}\label{prop:quant_est}
	Let $F$ be $s$-homogeneous for some $s\geq 2$ and $\mc A$-quasiaffine and let $p\in (s-1,\infty)$, $q\in (1,\infty)$ be such that  
	$\frac{s-1}{p} + \frac 1 q=1$. Given $v_1, v_2 \in C^\infty_{c, \mc A}(\R^n)$, we have the estimates
	\begin{equation}
	\label{eq:1stest}
	\left|\int_{\R^n} \varphi\left(F(v_1)-F(v_2)\right)\d x\right|\leq C \Vert v_1 - v_2\Vert_{\dot W^{-1,q}} \left(\Vert v_1\Vert_{L^p} + \Vert v_2 \Vert_{L^p} \right)^{s-1} \Vert \D \varphi \Vert_{L^\infty}
	\end{equation}
	and
	\begin{equation}
	\label{eq:2ndest}
	\left|\int_{\R^n} \varphi\left(F(v_1)-F(v_2)\right)\d x\right|\leq C \Vert v_1 - v_2\Vert_{\dot W^{-1,\BMO}} \left(\Vert v_1\Vert_{L^p} + \Vert v_2 \Vert_{L^p} \right)^{s-1} \Vert \D \varphi \Vert_{L^q}.
	\end{equation}
	for all $\varphi \in C^\infty_c(\R^n)$.
\end{prop}

We remark that (\ref{eq:1stest}) estimates the Kantorovich--Rubinstein--Wasserstein norm of the difference $F(v_1)-F(v_2)$. If we take $p=q=s$,  we recover a quantitative version of the statement
$$v_j \w v  \tp{ in } L^s_\mc A(\R^n)\hs \implies \hs F(v_j)\wstar F(v) \tp{ in } \mathscr D'(\R^n).$$
For the second estimate (\ref{eq:2ndest}), we define the norm in  $\dot W^{-1, \BMO}(\R^n)$ by
$$\Vert v\Vert_{\dot W^{-1,\BMO}}\equiv \left \Vert\mc F^{-1}\left(\frac{\widehat{v}(\xi)}{|\xi|}\right) \right \Vert_{\BMO}. $$
Furthermore, Proposition \ref{prop:quant_est} yields continuity
results in the regime below integrability:

\begin{remark}
In this remark we discuss the case $p<s$, so that the quantity $F(v)$
is not integrable, and we define an appropriate distributional version
of $F$.

Given a sequence $v_j\in C_{c,\mc A}^\infty(\R^n)$ such that
$\sup_{j}\|v_j\|_{L^p}<\infty$ and $v_j\rightarrow v \in
\dot{W}{^{-1,q}}$, since $F(v_j)\in C_c^\infty(\R^n)$, we can define 
\begin{equation}
F(v)\equiv \tp{w*-}\lim_{j\rightarrow\infty} F(v_j) \text{ in }
\mathscr{D}^\prime(\R^n);\label{eq:distributional}
\end{equation}
that this is well defined follows from estimate~\eqref{eq:1stest}.
A particularly relevant instance is the case when $v_j\in C_{c,\mc A}^\infty(\Omega)$ and $p>\frac{ns}{n+1}$, where $\Omega\subset\R^n$ is bounded and open. In this situation,
$$
v_j\rightharpoonup v\text{ in }L^p(\Omega) \hs \implies \hs \sup_{j}\|v_j\|_{L^p}<\infty\text{ and }v_j\rightarrow v \text{ in } \dot{W}{^{-1,q}}(\Omega),
$$
where the second convergence follows from the compactness of Sobolev
embeddings. In particular, (\ref{eq:1stest}) can be reinterpreted as a
weak continuity statement for the distributional version of $F$
defined in (\ref{eq:distributional}). Further properties of distributional null
Lagrangians will be explored
elsewhere \cite{GuerraRaitaSchrecker2020}.
\end{remark}

\begin{proof}[Proof of Proposition~\ref{prop:quant_est}]
	We use the strategy of the proof of Proposition \ref{prop:auxh1}, relying on the explicit structure of null Lagrangians. 
	We can write, by Proposition \ref{prop:hodge}, $v_i=\mc B u_i$. Let us first note that it suffices to prove (\ref{eq:1stest}) when $\mc B = \D^k$; indeed, assuming this has been done, and using Lemma \ref{lema:null_lags}, we have
	\begin{align*}
	\left|\int_{\R^n} \varphi\left(F(v_1)-F(v_2)\right)\d x\right|
	& = 
	\left|\int_{\R^n} \varphi\left(F\circ T(\D^k u_1)-F\circ T(\D^k u_2)\right)\d x\right|\\
	& \leq C \Vert \D^{k-1} (u_1-u_2)\Vert_{L^q} \left(\Vert \D^k u_1 \Vert_{L^p} + \Vert \D^k u_2 \Vert_{L^p} \right)^{s-1} \Vert \D \varphi \Vert_{L^\infty}\\
	& \leq C \Vert \mc B u_1- \mc B u_2\Vert_{\dot W^{-1,q}} \left(\Vert \mc B u_1\Vert_{L^p} + \Vert \mc B u_1 \Vert_{L^p} \right)^{s-1} \Vert \D \varphi \Vert_{L^\infty}
	\end{align*}
	which is precisely (\ref{eq:1stest}). In the last inequality we have used the fact that 
	$$ \Vert D^{k-1} u \Vert_{L^q} \leq C\left \Vert \mc F^{-1}\left(\frac{\widehat{\mc B u}(\xi)}{|\xi|}\right)\right \Vert_{L^q}\equiv \Vert \mc B u \Vert_{\dot W^{-1,q}} $$
	which follows from the identity
	$$\mc F\left(D^{k-1} u \right)(\xi) = \widehat u (\xi) \otimes \xi^{\otimes(k-1)} = \mc B^\dagger(\xi) \widehat{ \mc B u}(\xi) \otimes \xi^{\otimes (k-1)}= \mc B^\dagger\left(\frac{\xi}{|\xi|}\right)
	\frac{\widehat{\mc B u}(\xi)}{|\xi|}\otimes \left(\frac{\xi}{|\xi|}\right)^{\otimes (k-1)}$$
	together with the H\"ormander--Mihlin multiplier theorem. A similar argument shows that it also suffices to prove (\ref{eq:2ndest}) for $\mc B=\D^k$, by boundedness of Calderón-Zygmund operators from $\BMO$ to $\BMO$, see e.g. \cite[IV, \S 6.3a]{Stein2016}.
	
	Let us thus assume that $\mc B = \D^k$ and let us write $U_i \equiv \D^{k-1} u_i$. It suffices to consider the case where $F(\D^k u)=\det \D_{x^\prime}U^\prime (x)$, where we use the notation of the proof of Proposition \ref{prop:auxh1}.
	Note the algebraic identity
	\begin{equation}
	\label{eq:algebraicidentity}
	\det \D_{x'} U'_1-\det \D_{x'} U'_2 
	= \sum_{i,j=1}^s \frac{\p}{\p x'_i} \left[X^{(j)}_{ij}
	\left((U'_1)^j - (U'_2)^j\right)\right]\end{equation}
	where $X^{(j)}$ is the matrix 
	$$X^{(j)}\equiv \tp{cof}(\D_{x'} (U_2')^1, \dots, \D_{x'} (U_2')^{j-1}, \D_{x'} (U'_1-U'_2)^j, \D_{x'} (U'_1)^{j+1}, \dots, \D_{x'} (U'_1)^{s}).$$
	Then, integrating by parts and using H\"older's inequality, we get
	$$\left|\int_{\R^n} \varphi \left[F(\D U_1) - F(\D U_2)\right]\d x\right|\leq \sum_{j=1}^s \Vert (U'_1-U'_2)^j\Vert_{L^q} 
	\Vert \D U'_1\Vert_{L^p}^{s-j}\Vert \D U'_2\Vert_{L^p}^{j-1}
	\Vert \D \varphi\Vert_{L^\infty}$$
	from which the desired inequality
	$$\left|\int_{\R^n} \varphi \left[F(\D U_1) - F(\D U_2)\right]\d x\right|
	\leq \Vert U_1 - U_2 \Vert_{L^q} \left(\Vert \D U_1 \Vert_{L^p}+\Vert\D U_2 \Vert_{L^p}\right)^s \Vert \D \varphi \Vert_{L^\infty}
	$$
	follows.

	In order to prove (\ref{eq:2ndest}) for $\mc B=\D^k$ we use the Hardy space integrability of Proposition \ref{prop:auxh1}. Starting from (\ref{eq:algebraicidentity}), we do an integration by parts to get
	\begin{align*}
	\left|\int_{\R^n} \varphi \left(F(\D U_1) - F(\D U_2)\right)\d x \right|&
	\leq \sum_{j=1}^s \Vert (U'_1-U'_2)^j\Vert_\BMO \bigg\Vert \sum_{i=1}^s X_{ij}^{(j)} \frac{\p}{\p x'_i} \varphi \bigg\Vert_{\mathscr H^1}.
	\end{align*}
	Noting the estimate $$|X^{(j)}_{ij}|\leq |\D_{x'} (U_2')^1|\dots| \D_{x'} (U_2')^{j-1}| |\D_{x'} (U_1')^{j+1}|\dots| \D_{x'} (U_2')^{s}|,$$ we find, from the Hardy estimate and H\"older's inequality,
	$$\left|\int_{\R^n} \varphi \left[F(\D U_1) - F(\D U_2)\right]\d x\right|\leq \sum_{j=1}^s \Vert (U'_1-U'_2)^j\Vert_{\BMO} 
	\Vert \D U'_1\Vert_{L^p}^{s-j}\Vert \D U'_2\Vert_{L^p}^{j-1}
	\Vert \D \varphi\Vert_{L^q}$$
	from where one readily deduces (\ref{eq:2ndest}). 
\end{proof}

\appendix
\section{Computations for Proposition \ref{prop:nonunique}}

We shall only sketch the proof of the proposition, since it is purely
computational. The calculations are very involved and should be
performed with the help of symbolic computation software.
Let $n=3, \mb U = \R^7, \mb V=\R^7, \mb W=\R^3$ and consider the operator defined by
$$\mc A(\xi)=\begin{bmatrix}
\xi_1  & \xi_2 & \xi_3 & 0 & 0 & 0 & 0\\
0& 0&  0 &  \xi_1 & \xi_2 &\xi_3  & 0 \\
0 & \xi_1 & 0 & \xi_2 & 0 & 0 & \xi_3
\end{bmatrix}.$$
It is easy to see that $\tp{rank}\,\mc A(\xi)=3$ for all $\xi \neq 0$.

Our general strategy is as follows. Fix an order $k$ for $\mc B$ and consider a generic operator of that order: in other words, let
$\mc B(\xi)=\sum_{|\a|=k}\xi^\a B_\a $, where $B_\a=(b_\a^{i,j})$ are generic matrices with coefficients to be determined. One must have $\mc A(\xi)\mc B(\xi)=0$; this is a matrix whose entries are polynomials in $\xi$ and therefore is zero if and only if all coefficients of all the polynomials are zero. In other words, the condition $\mc A(\xi)\mc B(\xi)=0$ imposes a linear system on the variables $b^{i,j}_\a$. 

As a first step, one needs to verify that $\mc A$ does not admit potentials with first or second order.
For instance, when we look for potentials with order two, we can solve the system  
$$\mc A(\xi_1,\xi_2,0)\mc B(\xi_1,\xi_2,0)=0$$
to find that we must have $b^{i,j}_\a=0$ when $i=1,2,4,5$.
This shows that $\tp{rank\,}\mc B(\xi_1,\xi_2,0)\leq 3$, which cannot be if we are to have (\ref{eq:exact}).

However, $\mc A$ does have multiple cocanceling potentials of order three; they are quite complicated and the reader can find the expressions of two of them, $\mc B_1$ and $\mc B_2$, below. In order to verify that they are cocanceling, one can check for instance that, with $e_i$ being the canonical basis in $\R^3$,
$$\bigcap_{i=1}^3\ker \mc B_j(e_i)=\{0\} \hs \tp{ for } j=1,2.$$
One can also verify that there is no isomorphism $Q\in \tp{GL}(\mb U)$ such that $\mc B_1(\xi) Q = \mc B_2(\xi)$; this can be achieved by testing with $\xi=e_i$ for $i=1,2,3$ as above.

We denote by $(\mc B_i)_{\bullet j}$ the $j$-th column of $\mc
B_i$. Then we have, for $i=1$,

{\tiny
	\begin{gather*}
	(\mc B_1)_{\bullet 1}
	= \begin{bmatrix}
	-p_5-p_7 \\ p_4-p_{10} \\ p_2+p_9 \\ -p_2-p_3-p_5-p_6-2 p_8-p_{10} \\ p_1+p_5-p_8-p_{10} \\ p_1+p_2+p_3+p_4+p_6+p_7+p_9 \\ p_2+p_4+p_5+p_6+2 p_7+p_9 \\
	\end{bmatrix}, \hs
	(\mc B_1)_{\bullet 2}
	= \begin{bmatrix}
	-p_3-p_5-p_7-p_8-p_{10} \\ p_4-p_6+p_8-p_{10} \\ p_1+p_2+p_4+p_5+p_6-p_7+p_9 \\ -p_2-p_5-p_6-p_8-p_9-p_{10} \\ p_1+p_5-p_8-p_9 \\ p_2+p_3+p_5+p_6+p_7+p_8 \\ p_3+p_5+p_6+p_7+p_8+p_9 \\
	\end{bmatrix}
	\\
	(\mc B_1)_{\bullet 3}
	= \begin{bmatrix}
	p_5+p_8 \\ -p_3-p_5+p_8 \\ -p_7 \\ -p_5-p_8-2 p_9-p_{10} \\ p_3+p_6-p_8-p_9-p_{10} \\ p_4+p_5+p_6+p_7+p_8+p_9 \\ p_1+p_2+p_7+2 p_8+p_9 \\
	\end{bmatrix}
	\hs
	(\mc B_1)_{\bullet 4}
	= \begin{bmatrix}
	0 \\ -p_3+p_8+p_9 \\ p_2-p_7-p_8 \\ -p_3-p_5-p_6-2 p_8 \\ p_5-p_9-p_{10} \\ p_1+p_2+p_3+p_4+p_8+p_9 \\ p_1+p_2+2 p_7 \\
	\end{bmatrix}
	\\
	(\mc B_1)_{\bullet 5}
	= \begin{bmatrix}
	p_5-p_6-p_9 \\ -p_3+p_8-p_9 \\ p_3+p_5-p_7+p_8 \\ -p_3-p_5-p_{10} \\ p_3-p_9-p_{10} \\ p_1+p_6+p_8+p_9 \\ p_1+p_2+p_5+p_9 \\
	\end{bmatrix}
	\hs 
	(\mc B_1)_{\bullet 6}
	= \begin{bmatrix}
	p_8-p_3 \\ -p_3-p_5+p_8 \\ p_1+p_2-p_7 \\ -p_5-p_8-p_9 \\ 0 \\ p_2+p_4+p_5 \\ p_1+p_2+p_7+p_8 \\
	\end{bmatrix}
	\hs
	(\mc B_1)_{\bullet 7}
	= \begin{bmatrix}
	-p_3-p_7-2 p_8+p_9 \\ -p_3+p_4+p_5-p_6+p_9-p_{10} \\ p_1+p_2+p_4-p_8+p_9 \\ -p_2-p_3-p_5-p_6-2 \left(p_8+p_9\right)-p_{10} \\ p_1+p_3+p_5+p_6 \\ p_1+p_3+p_4+p_5+p_6 \\ p_1+p_3+p_4+p_6+2 \left(p_7+p_8\right)+p_9 \\ 
	\end{bmatrix}
	\end{gather*}
}

For $i=2$, we have

{\tiny
	\begin{gather*}
	(\mc B_2)_{\bullet 1}
	= \begin{bmatrix}
	-p_6-p_7-p_8+p_9-p_{10} \\ p_4-p_6-p_{10} \\ p_3+p_4+p_6+p_9 \\ -p_2-p_3-p_5-p_6-p_8-p_9 \\ p_1+p_3+p_5+p_6-p_8-p_{10} \\ p_1+p_3+p_7+p_9 \\ p_2+p_3+p_4+p_5+p_6+p_7+p_8
	\end{bmatrix}, \hs
	(\mc B_2)_{\bullet 2}
	= \begin{bmatrix}
	-p_3-p_5-p_6-p_7-p_8 \\ p_4-p_6-p_9-p_{10} \\ p_1+p_2+p_3+p_4+p_5+p_8+p_9 \\ -p_2-p_3-p_5 \\ p_1 \\ p_1+p_2 \\ p_2+p_3+p_4+p_5+p_6
	\end{bmatrix}
	\\
	(\mc B_2)_{\bullet 3}
	= \begin{bmatrix}
	-p_3-p_5-p_6+p_8 \\ -p_5-p_6-p_{10} \\ p_1+p_2+p_3+p_5+p_9 \\ -p_5-p_8-p_{10} \\ p_3+p_5-p_9-p_{10} \\ p_6+p_8+p_9 \\ p_2+p_3+p_4+p_6+p_7+p_9
	\end{bmatrix}
	\hs
	(\mc B_2)_{\bullet 4}
	= \begin{bmatrix}
	-p_3-p_6-p_7-p_{10} \\ p_4-p_6+2 p_8+p_9-p_{10} \\ p_1+p_3+p_5+p_6-2 p_7-p_8+p_9 \\ -p_2-2 p_5-p_6-p_8-2 p_9 \\ p_1+p_3+p_6-p_8-p_9-p_{10} \\ p_2+p_3+p_4+p_5+p_7+p_8+p_9 \\ p_3+p_6+p_7+2 p_8\end{bmatrix}
	\\
	(\mc B_2)_{\bullet 5}
	= \begin{bmatrix}
	p_5-p_7-p_8-p_{10} \\ -p_3+p_4-p_6+p_8+p_9 \\ p_4+p_5+p_6-p_7-p_8 \\ -p_2-p_5-p_6-2 p_9 \\ p_1+p_3+p_6-p_9 \\ p_3+p_5+p_8 \\ p_1+p_3+2 p_8
	\end{bmatrix}
	\hs 
	(\mc B_2)_{\bullet 6}
	= \begin{bmatrix}
	-p_7-p_8-p_{10} \\ -p_3+p_4-p_6+p_8-p_{10} \\ p_2+p_4+p_5+p_6-p_7+p_9 \\ -p_2-p_5-p_8-p_9 \\ p_1+p_3+p_5-p_8-p_{10} \\ p_5+p_7+p_9 \\ p_1+p_3+p_6+p_7+p_8 \end{bmatrix}
	\hs
	(\mc B_2)_{\bullet 7}
	= \begin{bmatrix}
	-p_6-p_7-p_8 \\ -p_3+p_4+p_8-p_9 \\ p_2+p_3+p_4-p_7+p_8 \\ -p_2-p_3-p_5-2 p_8-p_9 \\ p_1+p_5+p_6-p_8-p_{10} \\ p_1+p_2+p_4+p_7+p_9 \\ p_1+p_2+p_5+2 p_7+p_8 
	\end{bmatrix}
	\end{gather*}
}
where we made for simplicity the substitutions
$$\begin{matrix*}[r]
\xi _1^3= p_1, &\xi _1^2 \xi _2= p_2, & \xi _1^2 \xi _3= p_3, & \xi _1  \xi _2^2 = p_4, & 
\xi _1 \xi _2 \xi _3 = p_5 \\   \xi _1 \xi _3^2= p_6, & \xi _2^3 = p_7, &  \xi  _2^2 \xi_3 =p_8,
& \xi_2\xi_3^2=p_9,&\xi_3^3= p_{10}.
\end{matrix*}$$

{\footnotesize
\bibliographystyle{acm}

\bibliography{/Users/antonialopes/Dropbox/Oxford/Bibtex/library.bib}

}

\end{document}